\documentclass[12pt,a4paper]{article}

\usepackage[all]{xy}
\usepackage{geometry}
\usepackage[dvips]{graphicx}

\usepackage[english]{babel}
\usepackage[utf8]{inputenc}
\usepackage[normalem]{ulem}

\usepackage{latexsym,amsfonts,amssymb, amsmath, amsthm}
\usepackage{mathtools}
\usepackage{multicol}   
\usepackage[inline]{enumitem}
\usepackage{stackengine}

\renewcommand{\emptyset}{\ensuremath{\varnothing}}

\newcommand{\R}{\ensuremath{\mathbb R}}

\newcommand{\from}{:}
\newcommand{\id}{\ensuremath{\mathrm{id}}}
\newcommand{\abs}[1]{\ensuremath{\left\lvert #1\right\rvert}}

\newcommand{\cbullet}{\,\raisebox{1pt}{\(\scriptscriptstyle\bullet\)}\,}

\renewcommand{\implies}{\ensuremath{\Rightarrow}}
\renewcommand{\iff}{\ensuremath{\Leftrightarrow}}

\newtheorem{lemma}{Lemma}
\newtheorem{proposition}{Proposition}
\newtheorem{theorem}{Theorem}
\newtheorem{corollary}{Corollary}
\newtheorem{remark}{Remark}
\newtheorem{claim}{Claim}
\newtheorem{example}{Example}

\title{From Hölder triangles to the whole plane}
\author{Sergio Alvarez}
\date{}

\begin{document}

\maketitle

\begin{abstract}
We show how to determine whether two given real polynomial functions of a single variable are Lipschitz equivalent by comparing the values and also the multiplicities of the given polynomial functions at their critical points. Then we show how to reduce the problem of \({\cal R}\)-semialgebraic Lipschitz classification of \(\beta\)-quasihomogeneous polynomials of two real variables to the problem of Lipschitz classification of real polynomial functions of a single variable, under some fairly general conditions.
\end{abstract}

\section{Introduction}

In \cite{HP1}, Henry and Parusi\'{n}ski showed that the bi-Lipschitz classification of complex analytic function germs admits continuous moduli. This fact had not been observed before and interestingly it contrasts with the fact that the bi-Lipschitz equivalence of complex analytic set germs does not admit moduli \cite{M}. The moduli space of bi-Lipschitz equivalence is not yet completely understood but it is worth noting that recently Câmara and Ruas have made progress in the study of the moduli space of bi-Lipschitz equivalence of quasihomogeneous function germs, in the complex case \cite{CR}.

In \cite{HP2}, Henry and Parusi\'{n}ski showed that the bi-Lipschitz classification of real analytic function germs admits continuous moduli. The particular case of weighted homogeneous polynomial functions of two real variables has been considered by Koike and Parusi\'{n}ski in \cite{KP}. Then, independently, in \cite{BFP}, Birbrair, Fernandes and Panazzolo described the bi-Lipschitz moduli in the ``simplest possible case'' (as they have called it): quasihomogeneous polynomial functions defined on the Hölder triangle 
\(T_\beta \coloneqq \{(x,y)\in\R^2: 0\leq x, \,0\leq y\leq x^\beta\}\). 

In this paper, we consider the problem of classifying 
\(\beta\)-quasihomogeneous polynomials in two variables with real coefficients modulo \({\cal R}\)-semialgebraic Lipschitz equivalence.
Here and throughout the text, \(\beta\) always denotes a rational number 
\(> 1\). (We define \(\beta\)-quasihomogeneous polynomials and 
\({\cal R}\)-semialgebraic Lipschitz equivalence in section 
\ref{section: R-semialgebraic Lipschitz equivalence and height functions}.)
Our main goal is to extend the results obtained in \cite{BFP} for the classification of germs of functions defined on the Hölder triangle to germs of functions defined on the whole plane.

Following the strategy used in \cite{BFP}, we attack the problem by reducing it to the Lipschitz classification of real polynomial functions of a single variable. The classification of polynomial functions of a single variable is carried out in section \ref{section: Lipschitz equivalence on a single variable}. Here too, we follow the strategy used in \cite{BFP}, which consists in comparing the values and also the multiplicities of the given polynomial functions at their critical points.

In order to accomplish the reduction to the single variable case,
still following the strategy used in \cite{BFP}, we associate with each 
\(\beta\)-quasihomogeneous polynomial \(F(X,Y)\in\R[X,Y]\) a pair of 
polynomial functions \(f_+,f_-\from\R\to\R\), called the {\it height functions} of \(F\), which encode a great deal of information about the original polynomial.
Then, we consider the following questions:
\begin{enumerate}
\item Suppose that two given \(\beta\)-quasihomogeneous polynomials 
\(F,G\in\R[X,Y]\) of degree \(d\geq 1\) are \({\cal R}\)-semialgebraically Lipschitz equivalent. Is it possible to arrange their height functions in pairs of Lipschitz equivalent functions?
\item Suppose that the height functions of two given \(\beta\)-quasihomogeneous polynomials \(F,G\in\R[X,Y]\) of degree \(d\geq 1\) can be arranged in pairs of Lipschitz equivalent functions. Are \(F\) and \(G\)\/ 
\({\cal R}\)-semialgebraically Lipschitz equivalent?
\end{enumerate}

We show that if the zero sets of the polynomials \(F\) and \(G\) have points both on the right half-plane and on the left half-plane then the answer to the first question is yes (Corollary \ref{cor: equivalent polynomials, Lipschitz equivalent height functions}). Also, we obtain some rather general conditions under which the answer to question 2 is affirmative (Theorem \ref{thm: Sufficient conditions for R-semialg. Lip. equivalence}). These are the main results of the paper.

In \cite{BFP}, the questions stated above were both answered affirmatively in the case where the equivalence is restricted to the Hölder triangle \(T_\beta\), assuming that the given \(\beta\)-quasihomogeneous polynomials vanish identically on \(\partial T_\beta\) and do not vanish at the interior points of 
\(T_\beta\). Here, we generalize the methods devised by Lev Birbrair, Alexandre Fernandes, and Daniel Panazzolo. This generalization leads to the theory of \(\beta\)-transforms and inverse \(\beta\)-transforms presented in sections \ref{section: beta-isomorphisms and the beta-transform},
\ref{section: proto-transitions}, and \ref{section: beta-transitions and the inverse beta-transform}. In sections \ref{section: R-semialgebraic Lipschitz equivalence and height functions} and \ref{section: A characterization of beta-transitions}, we apply the general theory to obtain our main results.

\section{Lipschitz equivalence of polynomial functions of a single variable}
\label{section: Lipschitz equivalence on a single variable}

Two nonzero polynomial functions \(f,g\from\R\to\R\) are {\it Lipschitz equivalent}, written \(f \cong g\), if there exist a bi-Lipschitz homeomorphism \(\phi\from\R\to\R\) and a constant \(c > 0\) such that
\begin{equation}\label{eq: Lipschitz equivalence}
g\circ\phi = c f.
\end{equation}

In this section, we provide effective criteria to determine whether any two   nonzero polynomial functions \(f,g\from\R\to\R\) are Lipschitz equivalent. Clearly, two nonzero constant functions are Lipschitz equivalent if and only if they have the same sign; so we focus on nonconstant polynomial functions.  In this case, \(\phi\) is necessarily semialgebraic (see \cite[Lemma~3.1]{BFP}).

\begin{lemma}
\label{lemma: immediate consequences Lipschitz equivalence equation}
Let \(f,g\from\R\to\R\) be nonconstant polynomial functions. Suppose that \(f\) and \(g\) are Lipschitz equivalent, so that \(g\circ \phi = cf\), for some bi-Lipschitz function \(\phi\from\R\to\R\) and some constant \(c > 0\). We have:
\end{lemma}
\begin{enumerate}[label=\roman*.]
\item \(\deg f = \deg g\)
\item \(\lim_{t\to+\infty} \phi(t)/t = \lim_{t\to-\infty} \phi(t)/t\)
\item \(0 < \lim_{\abs{t}\to+\infty} \abs{\phi(t)/t} < \infty\)
\end{enumerate}

\begin{proof}
Let \(f(t) = \sum_{i=0}^d a_it^i\) and \(g(t) = \sum_{i=0}^e b_it^i\), where \(a_d,b_e\neq 0\), and let 
\begin{equation*}
l_+\coloneqq \lim_{t\to+\infty}\frac{\phi(t)}{t} 
\quad\text{and}\quad
l_-\coloneqq \lim_{t\to-\infty}\frac{\phi(t)}{t}.
\end{equation*}
These limits are both well-defined in the extended real line because \(\phi\) is semialgebraic. Also, since \(\phi\) is bi-Lipschitz, they are nonzero real numbers.

Since \(\lim_{\abs{t}\to+\infty}\abs{\phi(t)} = +\infty\), we have
\begin{equation*}
\lim_{\abs{t}\to+\infty}\frac{g(\phi(t))}{\phi(t)^e} 
= \lim_{\abs{t}\to+\infty}\frac{g(t)}{t^e}
= b_e.
\end{equation*}
Then,
\begin{equation*}
c\cdot\lim_{t\to+\infty}\frac{f(t)}{t^e} = \lim_{t\to+\infty}\frac{g(\phi(t))}{t^e}
=\lim_{t\to+\infty}\frac{g(\phi(t))}{\phi(t)^e}\cdot
   \lim_{t\to+\infty} \left(\frac{\phi(t)}{t}\right)^e = b_e l_+^e.
\end{equation*}
Since \(b_el_+^e\) is a nonzero real number, it follows that \(d = e\) 
(which proves (i)) and also that
\begin{equation*}
c\cdot a_d = c\cdot\lim_{t\to+\infty} f(t)/t^d = b_dl_+^d.
\end{equation*}
Similarly, we have:
\begin{equation*}
c\cdot a_d =  c\cdot\lim_{t\to-\infty} \frac{f(t)}{t^d}  
=\lim_{t\to-\infty}\frac{g(\phi(t))}{t^d}
=\lim_{t\to-\infty}\frac{g(\phi(t))}{\phi(t)^d}\cdot
   \lim_{t\to-\infty} \left(\frac{\phi(t)}{t}\right)^d 
= b_d l_-^d.
\end{equation*}
Thus, \(l_+^d = c\cdot a_d/b_d = l_-^d\), which implies that 
\(\abs{l_+} = \abs{l_-}\). 
On the other hand, since \(\phi\) is monotonic and 
\(\lim_{\abs{t}\to+\infty}\abs{\phi(t)} = +\infty\),
either \(\phi(t)\) and \(t\) have the same sign for large \(\abs{t}\), or 
\(\phi(t)\) and \(t\) have opposite signs for  large \(\abs{t}\). 
In any case, \(l_+\) and \(l_-\) have the same sign, and therefore \(l_+ = l_-\). This proves (ii) and ensures that the limit \(\lim_{\abs{t}\to+\infty} \phi(t)/t\) is a well-defined nonzero real number. Hence, (iii) is also proved.
\end{proof}

By Lemma \ref{lemma: immediate consequences Lipschitz equivalence equation}, equality of degrees is a necessary condition for Lipschitz equivalence. Thus, the problem of classifying polynomial functions modulo Lipschitz equivalence is reduced to the classification of nonconstant polynomial functions of the same degree.

\begin{lemma}
\label{lemma: bi-Lipschitz iff bi-analytic}
 Let \(f,g\from\R\to\R\) be two polynomial functions of the same degree 
 \(d\geq 1\), and suppose that \(\phi\from\R\to\R\) is a bijective function such that \(g\circ\phi = cf\), for some constant \(c > 0\). The following conditions are equivalent:
\begin{enumerate}[label=\roman*.]
\item \(\phi\) is bi-Lipschitz;
\item The multiplicity of \(f\) at \(t\) is equal to the multiplicity of \(g\) at \(\phi(t)\), for all \(t\in\R\); 
\item \(\phi\) is bi-analytic.
\end{enumerate}
\end{lemma}
 
\begin{proof}
\begin{description}[leftmargin=0cm]
\item[\((i)\implies (ii)\colon\)] Pick any point \(t_0\in\R\). Let \(k\) be the multiplicity of \(f\) at \(t_0\), and let \(l\) be the multiplicity of \(g\) at \(\phi(t_0)\). For any pair of functions \(u,v\from\R\to\R\), we write \(u \sim v\) to indicate that there exist constants \(A,B>0\) such that
\(A\abs{v(t)}\leq \abs{u(t)} \leq B\abs{v(t)}\), 
for \(t\) sufficiently close to \(t_0\). Then, \(f(t) - f(t_0) \sim (t-t_0)^k\) and \(g(s) - g(\phi(t_0)) \sim (s-\phi(t_0))^l\). Since \(g\circ \phi = c f\), this implies that \((\phi(t)-\phi(t_0))^l \sim (t-t_0)^k\). And since we are assumig that \(\phi\) is bi-Lipschitz, it follows that \((t-t_0)^l \sim (t-t_0)^k\). Therefore, \(k = l\).

\item[\((ii)\implies (iii)\colon\)] Pick any point \(t_0\in\R\). 
Suppose that \(\hat{f}\coloneqq c f\) has multiplicity \(k\) at \(t_0\).
Then, there exist an increasing analytic diffeomorphism 
\(u\from I \to (-\epsilon, \epsilon)\), with \(t_0\in I\), and a constant 
\(\rho\in\R\setminus\{0\}\), such that \(u(t_0) = 0\) and 
\(\hat{f}\circ u^{-1}(t) = a + \rho t^k\), for \(\abs{t} < \epsilon\);
where \(a \coloneqq \hat{f}(t_0) = g\circ\phi(t_0)\).

Since we are assuming that condition \((ii)\) holds, the multiplicity of \(g\) at the point \(\phi(t_0)\) is also \(k\). Then, as before, there exist an increasing analytic diffeomorphism \(v\from J \to (-\epsilon^\prime, \epsilon^\prime)\), with \(\phi(t_0)\in J\), and a constant \(\sigma\in\R\setminus\{0\}\), such that \(v(\phi(t_0)) = 0\) and 
\(g\circ v^{-1}(t) = a + \sigma t^k\), for \(\abs{t} < \epsilon^\prime\). 

Shrinking the interval \(I\), if necessary, we can assume that \(\phi(I)\subseteq J\). Hence, we can write \(\hat{f}\circ u^{-1}(t) = g\circ v^{-1}(\bar{\phi}(t))\), where 
\(\bar{\phi} \coloneqq 
v\circ \phi\circ u^{-1}\from (-\epsilon,\epsilon)\to (-\epsilon^\prime,\epsilon^\prime)\);
and then it follows that \(\bar{\phi}(t) = \nu t\), where
\(\nu = \pm\abs{\frac{\rho}{\sigma}}^{\frac{1}{k}}\), depending on whether 
\(\phi \) is increasing (positive sign) or decreasing (negative sign).
In particular, this shows that \(\bar{\phi}\) is analytic.

Therefore, \(\phi |_{I} = v^{-1}\circ \bar{\phi} \circ u\) is analytic; so \(\phi\) is analytic at \(t_0\). Since the point \(t_0\in\R\) is arbitrary, it follows that \(\phi\) is an analytic function. A similar argument, applied to \(\phi^{-1}\), shows that 
\(\phi^{-1}\) is also analytic.

\item[\((iii)\implies (i)\colon\)] Suppose that \(\phi\) is bi-analytic. 
Then, in particular, \(\phi\) is a homeomorphism, so \(\phi\) is monotonic and 
\(\lim_{\abs{t}\to+\infty}\abs{\phi(t)}~=~+\infty\). Also, since \(f,g\) are non-constant polynomial functions and \(g\circ \phi = cf\), \(\phi\) is a semialgebraic function (see \cite[Lemma~3.1]{BFP}).

Let \(f(t) = \sum_{i=0}^{d} a_i t^i\) and \(g(t) = \sum_{i=0}^{d} b_i t^i\), 
with \(a_d,b_d\neq 0\). Since \(g\circ\phi = cf\) and
\(\lim_{\abs{t}\to+\infty}\abs{\phi(t)}~=+\infty\), we have
\begin{equation}
\label{eq: limits to the power d}
\lim_{t\to-\infty}\left(\frac{\phi(t)}{t}\right)^d = 
\lim_{t\to+\infty}\left(\frac{\phi(t)}{t}\right)^d =
c\cdot\frac{a_d}{b_d}
\end{equation}

Let 
\(l_+ \coloneqq \lim_{t\to+\infty}\phi(t)/t\) and 
\(l_- \coloneqq \lim_{t\to-\infty}\phi(t)/t\).
Both of these limits are well-defined in the extended real line because \(\phi\) is semialgebraic. It follows from equation (\ref{eq: limits to the power d}) that we actually have \(l_+,l_-\in\R\setminus\{0\}\) and \(\abs{l_+} = \abs{l_-}\). (Notice that to obtain this last equality from equation 
(\ref{eq: limits to the power d}), we use the fact that \(d > 0\).)

By L'Hôpital's rule, 
\(\lim_{t\to+\infty} \phi'(t) = l_+\) and \(\lim_{t\to-\infty} \phi'(t) = l_-\). 
(The existence of these limits in the extended real line is guaranteed by the fact that \(\phi'\) is semialgebraic, so L'Hôpital's rule can be applied.)
Thus, \(\lim_{t\to+\infty} \abs{\phi'(t)} = \lim_{t\to-\infty} \abs{\phi'(t)}\); so that \(\abs{\phi'}\) can be continuously extended to a positive function defined on the compact space \(\R\cup\{\infty\}\cong S^1\). Hence, there exist constants \(A,B > 0\) such that \(A\leq \abs{\phi'(t)}\leq B\), for all \(t\in\R\). 
Therefore, \(\phi\) is bi-Lipschitz.
\end{description}
\end{proof} 

From the two lemmas above, it follows that if two polynomial functions 
\(f,g\from\R\to\R\) are Lipschitz equivalent, then they have the same degree and they have the same number of critical points. The first assertion is immediate from Lemma \ref{lemma: immediate consequences Lipschitz equivalence equation}. The second assertion follows from Lemma \ref{lemma: bi-Lipschitz iff bi-analytic}: if \(\phi\from\R\to\R\) is a bi-Lipschitz homemorphism such that \(g\circ \phi = cf\), for some constant \(c > 0\), then 
\(\phi\) induces a  1-1 correpondence between the critical points of \(f\) and the critical points of \(g\), because it preserves multiplicity. 

Propositions \ref{prop: no critical points}, \ref{prop: only one critical point}, and \ref{prop: at least 2 critical points} provide effective criteria to determine whether any two nonconstant polynomial functions \(f,g\from\R\to\R\), of the same degree, are Lipschitz equivalent. In Proposition \ref{prop: no critical points}, we consider the case in which \(f\) and \(g\)  have no critical points; 
in Proposition \ref{prop: only one critical point}, the case in which both 
\(f\) and \(g\) have only one critical point; and in Proposition 
\ref{prop: at least 2 critical points}, the case in which \(f\) and \(g\) have the same number \(k\geq 2\) of critical points.

\begin{proposition}
\label{prop: no critical points}
Let \(f,g\from\R\to\R\) be two polynomial functions of the same degree 
\(d\geq 1\). If \(f\) and \(g\) have no critical points, then \(f\) and \(g\) are Lipschitz equivalent. 
\end{proposition}

\begin{proof}
If \(f\) and \(g\) have no critical points, then they are both bi-analytic diffeomorphisms. Hence,  \(f = g\circ \phi\), where 
\(\phi \coloneqq g^{-1}\circ f\) is a bi-analytic diffeomorphism. By Lemma \ref{lemma: bi-Lipschitz iff bi-analytic}, \(\phi\) is bi-Lipschitz.
\end{proof}
 
\begin{proposition}
\label{prop: only one critical point}
Let \(f,g\from\R\to\R\) be two polynomial functions of the same degree 
\(d\geq 1\). Suppose that \(f\) has only one critical point \(t_0\), with multiplicity \(k\), and that \(g\) has only one critical point \(s_0\), with the same multiplicity \(k\). Also, suppose that \(f(t_0)\) and \(g(s_0)\) have the same sign\footnote{Clearly, this is a necessary condition for \(f\) and \(g\) to be Lipschitz equivalent.}(positive, negative or zero).
\begin{enumerate}[label=\roman*.]
\item If \(d\) is odd, then \(f\) and \(g\) are Lipschitz equivalent;
\item If \(d\) is even, then \(f\) and \(g\) are Lipschitz equivalent if and only if \(s_0\) and \(t_0\) are either both minimum points, or both maximum points  of \(f\) and \(g\), respectively.  
\end{enumerate}
\end{proposition}

\begin{proof}
First, consider the case where \(d\) is odd.
If a real polynomial function of a single variable, of odd degree, has only one critical point, then it is a homeomorphism. Thus, under the assumption that 
\(d\) is odd, \(f\) and \(g\) are homeomorphisms. Choose a constant \(c > 0\) such that \(g(s_0) = c f(t_0)\), and define \(\phi \coloneqq g^{-1}\circ \hat{f}\from\R\to\R\), where \(\hat{f} \coloneqq c f\). The function \(\phi\) is a bijection satisfying (\ref{eq: Lipschitz equivalence}), and the multiplicity of \(f\) at \(t\) is equal to the multiplicity of \(g\) at \(\phi(t)\) for all \(t\in\R\). By Lemma \ref{lemma: bi-Lipschitz iff bi-analytic}, \(\phi\) is bi-Lipschitz.

Now, suppose that \(d\) is even. If a  real polynomial function of a single variable, of even degree, has only one critical point, then this critical point is a point of local extremum (and consequently it is a point of even multiplicity).
If \(f\) and \(g\) are Lipschitz equivalent, then \(t_0\) and \(s_0\) are either both minimum points or both maximum points of \(f\) and \(g\), respectively. Otherwise, we would have \(\hat{f}(\R)\cap g(\R) = \{\hat{f}(t_0)\} = \{g(s_0)\}\), which is absurd, since equation (\ref{eq: Lipschitz equivalence}) implies that 
\(\hat{f}(\R) = g(\R) \). Conversely, suppose that \(t_0\) and \(s_0\) are either both minimum points or both maximum points of \(f\) and \(g\), respectively.
Pick any constant \(c > 0\) for which \(g(s_0) = c f(t_0)\), and define 
\(\phi\from\R\to\R\) by
\begin{equation*}
\phi |_{(-\infty,t_0]} \coloneqq 
     \left(g |_{(-\infty,s_0]} \right)^{-1}\circ \hat{f}|_{(-\infty,t_0]},
\quad
\phi |_{[t_0,+\infty)} \coloneqq 
     \left(g |_{[s_0,+\infty)} \right)^{-1}\circ \hat{f}|_{[t_0,+\infty)}.
\end{equation*}

Clearly, \(\phi\) is a bijection satisfying (\ref{eq: Lipschitz equivalence}), and the multiplicity of \(f\) at \(t\) is equal to the multiplicity of \(g\) at \(\phi(t)\) for all \(t\in\R\). Again, by Lemma \ref{lemma: bi-Lipschitz iff bi-analytic}, \(\phi\) is bi-Lipschitz.
\end{proof}

For the case in which \(f\) and \(g\) have the same number \(p\geq 2\) of critical points, we introduce an adapted version of the notion of {\it multiplicity symbol} defined in \cite{BFP}.

Let \(f\from\R\to\R\) be a polynomial function of degree \(d\geq 1\), having exactly \(p\) critical points, with \(p\geq 2\). Let \(t_1< \ldots < t_p\) be the critical points of \(f\), with multiplicities \(\mu_1,\ldots,\mu_p\), respectively. 
The {\it multiplicity symbol} of \(f\) is the ordered pair \(\left(a,\mu\right)\) whose first entry is the \(p\)-tuple \(a = \left(f(t_1),\ldots,f(t_p)\right)\), and second entry is the \(p\)-tuple \(\mu = \left(\mu_1,\ldots,\mu_p\right)\).

Let \(g\from\R\to\R\) be another polynomial function of degree \(d\geq 1\), having exactly the same number \(p\geq 2\) of critical points. 
Let \(s_1< \ldots < s_p\) be the critical points of \(g\), with multiplicities \(\nu_1,\ldots,\nu_p\), respectively. The multiplicity symbol of \(g\) is the ordered pair \(\left(b,\nu\right)\), where 
\(b = \left(g(s_1),\ldots,g(s_p)\right)\) and \(\nu = \left(\nu_1,\ldots,\nu_p\right)\).

The multiplicity symbols \(\left(a,\mu\right)\) and \(\left(b,\nu\right)\) are said to be: 
\begin{enumerate}[label=\roman*.]
\item {\it directly similar}, if there exists a constant \(c > 0\) such that \(b = c\cdot a\), and \(\nu = \mu\);
\item  {\it reversely similar}, if there exists a constant \(c > 0\) such that 
\(b = c\cdot \overline{a}\), and \(\nu=\overline{\mu}\).
\end{enumerate}
For any \(p\)-tuple \(x = \left(x_1,\ldots,x_p\right)\), 
\(\overline{x} := \left(x_p,\ldots,x_1\right)\) is the \(p\)-tuple \(x\) written in reverse order.

The multiplicity symbols \(\left(a,\mu\right)\) and \(\left(b,\nu\right)\) are said to be {\it similar} if they are either directly similar or reversely similar.

\begin{proposition}
\label{prop: at least 2 critical points}
 Let \(f,g\from\R\to\R\) be two polynomial functions of the same degree 
\(d\geq 1\) having the same number \(p\geq 2\) of critical points. Then, \(f\) and \(g\) are Lipschitz equivalent if and only if their multiplicity symbols are similar.
\end{proposition}

\begin{proof}
First, suppose that \(f\) and \(g\) are Lipschitz equivalent. Then there exist a bi-Lipschitz function \(\phi\from\R\to\R\) and a constant \(c > 0\) such that
\(g\circ \phi = cf\). Let \(t_1,\ldots,t_p\) be the critical points of \(f\), with multiplicities \(\mu_1,\ldots,\mu_p\), and let \(s_1,\ldots,s_p\) be the critical points of \(g\), with multiplicities \(\nu_1,\ldots,\nu_p\). (As we noted just after the proof of Lemma \ref{lemma: bi-Lipschitz iff bi-analytic}, Lipschitz equivalent polynomials have the same number of critical points.)
Let \((a,\mu)\) be the multiplicity symbol of \(f\) and \((b,\nu)\) the multiplicity symbol of \(g\).

By Lemma \ref{lemma: bi-Lipschitz iff bi-analytic}, \(\phi\) preserves multiplicities. Thus, if \(\phi\) is increasing, then we have 
\(\phi(t_i) = s_i\) and \(\mu_i = \nu_i\), for \(i=1,\ldots,p\). Since \(g\circ \phi = cf\), it follows that also \(b_i = g(s_i) = c\cdot f(t_i) = c\cdot a_i\), for \(i=1,\ldots,p\). Hence, if \(\phi\) is increasing then the multiplicity symbols of \(f\) and \(g\) are directly similar. On the other hand, if \(\phi\) is decreasing, then we have \(\phi(t_{p+1-i}) = s_i\) and \(\mu_{p+1-i} = \nu_i\), for \(i=1,\ldots,p\). Since \(g\circ \phi = cf\), it follows that also \(b_i = g(s_i) = c\cdot f(t_{p+1-i}) = c\cdot a_{p+1-i}\), for \(i=1,\ldots,p\). Hence, if \(\phi\) is decreasing then the multiplicity symbols of \(f\) and \(g\) are reversely similar.
In any case, it follows that the multiplicity symbols of \(f\) and \(g\) are similar. 

Now we prove the converse. Suppose that the multiplicity symbols of \(f\) and \(g\) are similar. Replacing \(g\) with \(g\circ \alpha\), where \(\alpha \coloneqq -\id\from\R\to\R\) (which is bi-Lipschitz), we can assume that the multiplicity symbols of \(f\) and \(g\) are directly similar. 

Let \(t_1< \ldots < t_p\) be the critical points of \(f\), with multiplicities \(\mu_1,\ldots,\mu_p\), respectively; and let \(s_1< \ldots < s_p\) be the critical points of \(g\), with multiplicities \(\nu_1,\ldots,\nu_p\), respectively. Since we are assuming that the multiplicity symbols of \(f\) and \(g\) are directly similar, there exists a constant \(c > 0\) such that \(g(s_i) = c f(t_i)\), for \(i =1,\ldots,p\); and \(\mu_i = \nu_i\), for \(i =1,\ldots,p\). 

Let \(\hat{f}\coloneqq cf\) and \(c_i \coloneqq g(s_i) = \hat{f}(t_i)\).  
(Notice that \(c_i\neq c_{i+1}\), 
for \(1~\leq~i~<~k\).)
The functions \(\hat{f} |_{[t_i,t_{i+1}]}\from [t_i,t_{i+1}]\to [c_i,c_{i+1}]\) and
\(g |_{[s_i,s_{i+1}]}\from [s_i,s_{i+1}]\to [c_i,c_{i+1}]\) 
are both monotonic and injective.
The same is true for the functions 
\(\hat{f} |_{(-\infty,t_1]}\) and \(g |_{(-\infty,s_1]}\),
and also for the functions 
\(\hat{f} |_{[t_p,+\infty)}\) and \(g |_{[s_p,+\infty)}\).

Moreover, the functions 
\(\hat{f} |_{(-\infty,t_1]}\) and \(g |_{(-\infty,s_1]}\)
are either both increasing or both decreasing because
\(\hat{f}(t_1) = g(s_1)\), \(\hat{f}(t_2) = g(s_2)\), and \(\mu_1 = \nu_1\).
Since \(\hat{f}(t_1) = g(s_1)\) and 
\(\abs{\hat{f}(t)}, \abs{g(t)}\to+\infty\), as \(\abs{t}\to +\infty\),
this implies that
\(\hat{f}((-\infty,t_1]) = g((-\infty,s_1])\).

Similarly, the functions
\(\hat{f}|_{[t_p,+\infty)}\) and \(g |_{[s_p,+\infty)}\)
are either both increasing or both decreasing because
\(\hat{f}(t_{p-1}) = g(s_{p-1})\), \(\hat{f}(t_p) = g(s_p)\), and \(\mu_p = \nu_p\).
Since \(\hat{f}(t_p) = g(s_p)\), and
\(\abs{\hat{f}(t)}, \abs{g(t)}~\to~+\infty\), as \(\abs{t}\to+\infty\),
this implies that \(\hat{f}([t_p,+\infty)) = g([s_p,+\infty))\).

Define \(\phi\from\R\to\R\) by
\begin{align*}
\phi |_{(-\infty,t_1]} &\coloneqq 
     \left(g |_{(-\infty,s_1]}\right)^{-1}\circ \hat{f} |_{(-\infty,t_1]}\\
\phi |_{[t_i,t_{i+1}]} &\coloneqq 
     \left(g |_{[s_i,s_{i+1}]}\right)^{-1}\circ \hat{f} |_{[t_i,t_{i+1}]}, 
     \text{ for } 1\leq i < p\\
\phi |_{ [t_p,+\infty)} &\coloneqq 
     \left(g |_{[s_p,+\infty)}\right)^{-1}\circ \hat{f} |_{[t_p,+\infty)}
\end{align*}

Clearly, \(\phi\) is a bijection satisfying (\ref{eq: Lipschitz equivalence}), and
it takes \(t_i\) to \(s_i\), for \(i=1,\ldots,p\). 
Since the multiplicity symbols of \(f\) and \(g\) are directly similar, it follows that the multiplicity of \(f\) at \(t\) is equal to the multiplicity of \(g\) at \(\phi(t)\) for all \(t\in\R\). By Lemma \ref{lemma: bi-Lipschitz iff bi-analytic}, \(\phi\) is bi-Lipschitz.
\end{proof}

\section{Transformation of paths by Lipschitz maps}

In this section, we state and prove some results that will be used in the sequel.

\begin{lemma}
\label{lemma: finite initial velocity}
Let \(\Phi\from(\R^2,0)\to(\R^2,0)\) be a germ of  semialgebraic Lipschitz map, let 
\(\gamma\from[0,\epsilon)\to\R^2\)
be a continuous semialgebraic path such that \(\gamma(0) = 0\),
and let \(\tilde\gamma(t)\coloneqq \Phi(\gamma(t))\).
If \(\gamma\) has finite initial velocity then \(\tilde\gamma\) also has finite initial velocity, that is,  
if \,\(\lim_{t\to 0^+} \abs{\gamma(t)/t} < \infty\)
then \(\lim_{t\to 0^+} \abs{\tilde\gamma(t)/t} < \infty\).
\end{lemma}

\begin{proof}
Since \(\Phi\) is Lipschitz and \(\Phi(0) = 0\), there exists \(K > 0\) such that
\(\abs{\Phi(x,y)}\leq K \abs{(x,y)}\). Then,
\begin{equation*}
\abs{\frac{\tilde\gamma(t)}{t}} = \abs{\frac{\Phi(\gamma(t))}{t}} \leq
K \abs{\frac{\gamma(t)}{t}},
\end{equation*}
which implies that \(\tilde\gamma_+^\prime(0)\) is finite, given that 
\(\gamma_+^\prime(0)\) is finite, by hypothesis.
\end{proof}

\begin{lemma}
\label{lemma: transformation of paths with the same velocity}
Let \(\Phi\from(\R^2,0)\to(\R^2,0)\) be a germ of  semialgebraic Lipschitz map.
Let \(\gamma_1,\gamma_2\from[0,\epsilon)\to\R^2\) be two continuous semialgebraic paths, with finite initial velocity, 
such that \(\gamma_1(0) = \gamma_2(0) = 0\), and let 
\(\tilde\gamma_i(t)\coloneqq \Phi(\gamma_i(t)),\, i = 1,2\). 
If \(\gamma_{1^+}^\prime(0) = \gamma_{2^+}^\prime(0)\), 
then \(\tilde\gamma_{1^+}^\prime(0) = \tilde\gamma_{2^+}^\prime(0)\).
\end{lemma}

\begin{proof}
Since \(\Phi\) is Lipschitz and \(\Phi(0) = 0\), there exists \(K > 0\) such that
\(\abs{\Phi(x,y)}\leq K \abs{(x,y)}\). 
Since
\(\abs{\tilde\gamma_1(t) - \tilde\gamma_2(t)} 
\leq K\abs{\gamma_1(t) - \gamma_2(t)}\),
we have:
\begin{equation*}
\abs{\frac{\tilde\gamma_1(t)}{t} - \frac{\tilde\gamma_2(t)}{t}}
\leq K \abs{\frac{\gamma_1(t)}{t} - \frac{\gamma_2(t)}{t}},
\quad\text{ for } t  > 0.
\end{equation*}
Also, from \(\lim_{t\to 0^+} \gamma_i(t)/t = \gamma_{i^+}^\prime(0)\)
and \(\gamma_{1^+}^\prime(0) = \gamma_{2^+}^\prime(0)\), we see that
\begin{equation*}
\lim_{t\to 0^+} \abs{\frac{\gamma_1(t)}{t} - \frac{\gamma_2(t)}{t}} = 0,
\end{equation*}
since both \(\gamma_1\) and \(\gamma_2\) have finite initial velocity.

Then, by the Squeeze Theorem,
\begin{equation*}
\lim_{t\to 0^+} \abs{\frac{\tilde\gamma_1(t)}{t} - \frac{\tilde\gamma_2(t)}{t}} = 0.
\end{equation*}
By Lemma \ref{lemma: finite initial velocity},  
both \(\tilde\gamma_{1}\) and \(\tilde\gamma_{2}\) 
have finite initial velocity, so it follows that 
\(\tilde\gamma_{1^+}^\prime(0) = \tilde\gamma_{2^+}^\prime(0)\).
\end{proof}

\begin{corollary}
\label{cor: transformation of paths with the same direction}
Let \(\Phi\from(\R^2,0)\to(\R^2,0)\) be a germ of  semialgebraic Lipschitz map,
let \(\gamma_1,\gamma_2\from[0,\epsilon)\to\R^2\) be two continuous semialgebraic paths, with finite initial velocity, 
such that \(\gamma_1(0) = \gamma_2(0) = 0\), and let 
\(\tilde\gamma_i(t)\coloneqq \Phi(\gamma_i(t)),\, i = 1,2\).
If \(\gamma_{2^+}^\prime(0) = c\cdot\gamma_{1^+}^\prime(0)\), with \(c > 0\),
then \(\tilde\gamma_{2^+}^\prime(0) = c\cdot\tilde\gamma_{1^+}^\prime(0)\).
\end{corollary}

\begin{proof}
Consider the path 
\(\gamma_0(t)\coloneqq \gamma_1(ct)\), \(0\leq t < \epsilon/c\), 
and let \(\tilde\gamma_0(t)\coloneqq\Phi(\gamma_0(t))\).
By Lemma \ref{lemma: transformation of paths with the same velocity}, 
since \(\gamma_{0^+}^\prime(0) = c \gamma_{1^+}^\prime(0) = 
\gamma_{2^+}^\prime(0)\), we have 
\(\tilde\gamma_{0^+}^\prime(0)  = \tilde\gamma_{2^+}^\prime(0)\).
On the other hand, 
\(\tilde\gamma_0(t) = \Phi(\gamma_0(t)) = \Phi(\gamma_1(ct)) = 
\tilde\gamma_1(ct)\), so 
\(\tilde\gamma_{0^+}^\prime(0) = c \tilde\gamma_{1^+}^\prime(0)\).
Hence, 
\(\tilde\gamma_{2^+}^\prime(0) = c\cdot\tilde\gamma_{1^+}^\prime(0)\).
\end{proof}

\begin{corollary}
\label{cor: initial velocities with the same direction (bi-Lipschitz transformation)}
Let \(\Phi\from(\R^2,0)\to(\R^2,0)\) be a germ of  semialgebraic bi-Lipschitz map.
Let \(\gamma_1,\gamma_2\from[0,\epsilon)\to\R^2\) be two continuous semialgebraic paths, with finite initial velocity, 
such that \(\gamma_1(0) = \gamma_2(0) = 0\), and let 
\(\tilde\gamma_i(t)\coloneqq \Phi(\gamma_i(t)),\, i = 1,2\).
The initial velocities of the paths \(\gamma_1\) and \(\gamma_2\) have the same direction if and only if the initial velocities of the paths 
\(\tilde\gamma_1\) and \(\tilde\gamma_2\) have the same direction.
\end{corollary}

\section{\(\beta\)-isomorphisms and the \(\beta\)-transform}
\label{section: beta-isomorphisms and the beta-transform}

A germ of semialgebraic bi-Lipschitz map 
\(\Phi = (\Phi_1,\Phi_2)\from (\R^2,0)\to(\R^2,0)\) is said to be a 
{\it \(\beta\)-isomorphism of degree \(d\geq 1\)} if the following conditions are satisfied:
\begin{enumerate}[label = \roman*.]
\item There exist \(\beta\)-quasihomogeneous polynomials 
\(F(X,Y)\) and \(G(X,Y)\) of degree \(d\) such that 
\(G\circ \Phi = F\).

\item \(\lim_{x\to 0^+}\Phi_1(x,0)/x\neq 0\) \ and 
\ \(\lim_{x\to 0^-}\Phi_1(x,0)/x\neq 0\)
\end{enumerate}
When there is no danger of confusion, we shall simply say \(\beta\)-isomorphism, omitting the reference to the degree \(d\), 
which is always assumed to be \(\geq 1\). 

\begin{remark}
\label{rk: finite nonzero initial velocity}
For any germ of  semialgebraic bi-Lipschitz homeomorphism 
\(\Phi\from(\R^2,0)\to(\R^2,0)\), the path \(\Phi(x,0)\), \(0 \leq x < \epsilon\), has finite non-zero initial velocity. 

In fact, it is immediate from Lemma \ref{lemma: finite initial velocity} that the path \(\Phi(x,0)\), \(0 \leq x < \epsilon\), has finite initial velocity. On the other hand, there exists \(A > 0\) such that \(\Phi(x,0)\geq A\abs{x}\),  because 
\(\Phi\) is bi-Lipschitz and \(\Phi(0) = 0\); so we have 
\(\abs{\Phi(x,0)/x} \geq A\), for \(x\neq 0\). 
Hence, \(\lim_{x\to 0^+} \Phi(x,0)/x \neq 0\).

Similarly, we can prove that the path \(\Phi(-x,0)\), \(0 \leq x < \epsilon\), also has finite non-zero initial velocity. 
\end{remark}

\begin{remark}
\label{rk: equality of initial velocities}
Let \(\Phi\from (\R^2,0)\to(\R^2,0)\) be any germ of semialgebraic bi-Lipschitz map. It is immediate from Lemma \ref{lemma: transformation of paths with the same velocity} that for all \(t\in\R\), we have:
\begin{enumerate}[label=\roman*.]
\item The initial velocity of the path 
\(\Phi(x,tx^\beta)\), \(0\leq x < \epsilon\),
is equal to the initial velocity of the path
\(\Phi(x,0)\), \(0\leq x <\epsilon\).
\item The initial velocity of the path 
\(\Phi(-x,tx^\beta)\), \(0\leq x < \epsilon\),
is equal to the initial velocity of the path
\(\Phi(-x,0)\), \(0\leq x <\epsilon\).
\end{enumerate}

Hence, for all \(t\in\R\), 
\begin{equation}
\label{eq: equality of first components of initial velocities}
\lim_{x\to 0^+}\frac{\Phi_1(x,tx^\beta)}{x} = 
\lim_{x\to 0^+}\frac{\Phi_1(x,0)}{x}
\quad\text{and}\quad
\lim_{x\to 0^+}\frac{\Phi_1(-x,tx^\beta)}{x} = 
\lim_{x\to 0^+}\frac{\Phi_1(-x,0)}{x}\,.
\end{equation}

Also, for all \(t\in\R\), 
\begin{equation}
\label{eq: equality of second components of initial velocities}
\lim_{x\to 0^+}\frac{\Phi_2(x,tx^\beta)}{x} = 
\lim_{x\to 0^+}\frac{\Phi_2(x,0)}{x}
\quad\text{and}\quad
\lim_{x\to 0^+}\frac{\Phi_2(-x,tx^\beta)}{x} = 
\lim_{x\to 0^+}\frac{\Phi_2(-x,0)}{x}\,.
\end{equation}
\end{remark}

\begin{proposition}
\label{prop: order of growth of the second coordinate of a beta-isomorphism}
Let \(\Phi\from (\R^2,0)\to(\R^2,0)\) be a germ of semialgebraic bi-Lipschitz homeomorphism. Suppose that there exist \(\beta\)-quasihomogeneous polynomials \(F(X,Y)\) and \(G(X,Y)\) of degree \(d\geq 1\) such that 
\(G\circ \Phi = F\). We have:
\begin{enumerate}[label=\roman*.]
\item If \(\Phi\) is a \(\beta\)-isomorphism then, for each \(t\in\R\), 
\(\Phi_2(x,t\abs{x}^\beta) = O(\abs{x}^\beta)\) as \(x\to 0\).
\item If there exist \(t_0,t_1\in\R\) such that
\begin{equation*}
\Phi_2(x,t_0\abs{x}^\beta) = O(\abs{x}^\beta) \text{ as \(x\to 0^+\)}
\quad{and}\quad
\Phi_2(x,t_1\abs{x}^\beta) = O(\abs{x}^\beta) \text{ as \(x\to 0^-\),} 
\end{equation*}
then \(\Phi\) is a \(\beta\)-isomorphism.
\end{enumerate}
\end{proposition}

\begin{proof}
Suppose that \(\Phi\) is a \(\beta\)-isomorphism.
Let \(\beta = r/s\), with \(r > s > 0\), and \(\gcd(r,s) = 1\). 
Then we have
\begin{equation*}
G(X,Y) = \sum_{k=0}^m c_k X^{d-rk}Y^{sk},
\end{equation*}
where the coefficients \(c_k\) are real numbers, \(c_m\neq 0\), and 
\(m\leq\lfloor{d/r}\rfloor\).

By hypothesis, \(G(\Phi(x,y)) = F(x,y)\). Thus, for any \(t\in\R\) and \(x \neq 0\) sufficiently small,
\begin{equation*}
G(\Phi(x,t \abs{x}^\beta)) = F(x,t \abs{x}^\beta)\,.
\end{equation*}
Since the polynomials \(F\) and \(G\) are \(\beta\)-quasihomogeneous, this implies that
\begin{equation*}
G\left(\frac{\Phi_1(x,t\abs{x}^\beta)}{\abs{x}},
          \frac{\Phi_2(x,t\abs{x}^\beta)}{\abs{x}^\beta}\right) = f(t),
\end{equation*}
where
\begin{equation*}
f(t) = 
\begin{cases}
F(1,t), & \text{ if } \,x > 0\\
F(-1,t), & \text{ if } \,x < 0
\end{cases}
\ .
\end{equation*}

Hence, for each \(t\in\R\) and \(x\neq 0\) sufficiently small, 
\(y = \Phi_2(x,t \abs{x}^\beta)/\abs{x}^\beta\)
is a zero of the nonconstant polynomial
\begin{equation*}
H_{t,x}(y) \coloneqq G(\tilde x, y) - f(t) \in \R[y],
\end{equation*}
where \(\tilde x \coloneqq \Phi_1(x,t \abs{x}^\beta)/\abs{x}\).
Applying Cauchy's bound on the roots of a polynomial, we obtain
\begin{equation}
\label{eq: Cauchy's bound for Phi_2}
\abs{\frac{\Phi_2(x,t \abs{x}^\beta)}{\abs{x}^\beta}} \leq
1 + \max\left\{ 
     \abs{\frac{c_{m-1}}{c_m}}\abs{\tilde x}^r, \ldots, 
     \abs{\frac{c_{1}}{c_m}}\abs{\tilde x}^{r(m-1)},
     \abs{\frac{c_0{\tilde x}^d - f(t)}{c_m{\tilde x}^{d - rm}}}
     \right\} \,.
\end{equation}

Since \(\Phi\) is a \(\beta\)-isomorphism, we have
\begin{equation*}
\lim_{x\to 0^+}\frac{\Phi_1(x,0)}{x}\neq 0
\quad\text{and}\quad
\lim_{x\to 0^+}\frac{\Phi_1(-x,0)}{x}\neq 0
\end{equation*}

Then, by equation (\ref{eq: equality of first components of initial velocities}), we obtain
\begin{equation}
\label{eq: nonzero x tilde}
\lim_{x\to 0^+}\frac{\Phi_1(x,t\abs{x}^\beta)}{\abs{x}}\neq 0
\quad\text{and}\quad
\lim_{x\to 0^-}\frac{\Phi_1(x,t\abs{x}^\beta)}{\abs{x}}\neq 0,
\quad\text{for all }t\in\R.
\end{equation}

From (\ref{eq: Cauchy's bound for Phi_2}) and (\ref{eq: nonzero x tilde}), it follows that, for each \(t\in\R\), \(\Phi_2(x,t \abs{x}^\beta)/\abs{x}^\beta\) is bounded for \(x\neq 0\) sufficiently small. This proves \((i)\).

Now, suppose that for certain \(t_0, t_1\in\R\), 
\(\Phi_2(x,t_0 \abs{x}^\beta)/\abs{x}^\beta\) 
is bounded for \(x > 0\) sufficiently small,
and \(\Phi_2(x,t_1 \abs{x}^\beta)/\abs{x}^\beta\) 
is bounded for \(x < 0\) sufficiently small.
Then 
\begin{equation*}
\lim_{x\to 0^+}\frac{\Phi_2(x, t_0 x^\beta)}{x} = 
\lim_{x\to 0^+}
          \frac{\Phi_2(x, t_0 x^\beta)}{x^\beta} \cdot x^{\beta - 1}
          = 0
\end{equation*}
and
\begin{equation*}
\lim_{x\to 0^+}\frac{\Phi_2(-x, t_1 x^\beta)}{x} = 
\lim_{x\to 0^+}
          \frac{\Phi_2(-x, t_1 x^\beta)}{x^\beta} \cdot x^{\beta - 1}
          = 0
\end{equation*}
Thus, by (\ref{eq: equality of second components of initial velocities}), \begin{equation*}
\lim_{x\to 0^+}\frac{\Phi_2(x, 0)}{x} = 0
\quad\text{and}\quad
\lim_{x\to 0^+}\frac{\Phi_2(-x, 0)}{x} = 0\, .
\end{equation*}

On the other hand, by Remark \ref{rk: finite nonzero initial velocity},
the paths 
\(\Phi(x,0)\), \(0\leq x <\epsilon\), 
and
\(\Phi(-x,0)\), \(0\leq x <\epsilon\),
both have (finite) non-zero initial velocity.
Hence, 
\begin{equation*}
\lim_{x\to 0^+}\frac{\Phi_1(x,0)}{x} \neq 0
\quad\text{and}\quad
\lim_{x\to 0^+}\frac{\Phi_1(-x,0)}{x} \neq 0\, .
\end{equation*}
Therefore, \(\Phi\) is a \(\beta\)-isomorphism.
\end{proof}

\begin{proposition}
\label{prop: opposite horizontal initial velocities (beta-isomorphism)}
If \(\Phi\from (\R^2,0)\to(\R^2,0)\) is a \(\beta\)-isomorphism then the initial velocities of the paths \(\Phi(x,0)\), \(0\leq x < \epsilon\), and \(\Phi(-x,0)\), \(0\leq x < \epsilon\), are horizontal\/\footnote{We say that a vector 
\((v_1,v_2)\in\R^2\) is {\it horizontal} if \(v_1\neq 0\) and \(v_2 = 0\).} and have opposite directions.
\end{proposition}

\begin{proof}
Let \(\Phi\) be a \(\beta\)-isomorphism. By Proposition \ref{prop: order of growth of the second coordinate of a beta-isomorphism}, we have
\(\Phi_2(x,0) = O(\abs{x}^\beta)\) as \(x\to 0\). Then, 
\(\lim_{x\to 0^+} \Phi_2(x,0)/x = \lim_{x\to 0^+} \Phi_2(-x,0)/x = 0\). 
Since the paths \(\Phi(x,0)\), \(0\leq x < \epsilon\), and 
\(\Phi(-x,0)\), \(0\leq x < \epsilon\), both have nonzero finite initial velocity
(Remark \ref{rk: finite nonzero initial velocity}), it follows that both of them have horizontal initial velocity. By Corollary \ref{cor: initial velocities with the same direction (bi-Lipschitz transformation)}, the initial velocities of the paths
\(\Phi(x,0)\), \(0\leq x < \epsilon\), and \(\Phi(-x,0)\), \(0\leq x < \epsilon\), do not have the same direction. Since they are both horizontal, they have opposite directions. 
\end{proof}

It follows from Proposition \ref{prop: opposite horizontal initial velocities (beta-isomorphism)} that each \(\beta\)-isomorphism \(\Phi\from(\R^2,0)\to(\R^2,0)\) satisfies one of the following conditions:
\begin{enumerate}[label = (\Alph*)]
\item \label{condition: direct beta-isomorphism}
\(\lim_{x\to 0^+} \Phi_1(x,0)/x > 0\)
\,and \,\,\(\lim_{x\to 0^-} \Phi_1(x,0)/x > 0\)
\item \label{condition: reverse beta-isomorphism}
 \(\lim_{x\to 0^+} \Phi_1(x,0)/x < 0\)
\,and \,\,\(\lim_{x\to 0^-} \Phi_1(x,0)/x < 0\)
\end{enumerate}

A \(\beta\)-isomorphism \(\Phi\) is said to be {\it direct} if it satisfies 
\ref{condition: direct beta-isomorphism}, and it is said to be {\it reverse}
if it satisfies \ref{condition: reverse beta-isomorphism}.

\begin{proposition}
\label{prop: beta-isomorphism, horizontal initial velocity}
Let \(\Phi\from (\R^2,0)\to(\R^2,0)\) be a germ of semialgebraic bi-Lipschitz homeomorphism. Suppose that there exist \(\beta\)-quasihomogeneous polynomials \(F(X,Y)\) and \(G(X,Y)\) of the same degree \(d\geq 1\) such that \(G\circ \Phi = F\). The following assertions are equivalent:
\begin{enumerate}[label=\roman*.]
\item \(\Phi\) is a \(\beta\)-isomorphism
\item For any continuous semialgebraic path 
\(\gamma\from[0,\epsilon)\to \R^2\), if the initial velocity of \(\gamma\) is horizontal, then the initial velocity of 
\(\tilde\gamma \coloneqq \Phi\circ\gamma\) is horizontal.
\end{enumerate}
\end{proposition}

\begin{proof}
First, we prove that \((i) \implies (ii)\). Suppose that \(\Phi\) is a 
\(\beta\)-isomorphism. 

Take any continuous semialgebraic path \(\gamma\from [0,\epsilon)\to \R^2\) whose initial velocity is horizontal. Since the initial velocities of the paths
\(\alpha_{+}, \alpha_{-}\from [0,\epsilon) \to \R^2\), given by 
\(\alpha_{+}(t) = (t,0)\) and \(\alpha_{-}(t) = (-t,0)\), 
are both horizontal and have opposite directions, the initial velocity of 
\(\gamma\) has either the same direction as the initial velocity of \(\alpha_{+}\), or the same direction as the initial velocity of \(\alpha_{-}\).

By Corollary \ref{cor: initial velocities with the same direction (bi-Lipschitz transformation)}, this implies that the initial velocity of \(\tilde\gamma \coloneqq \Phi\circ\gamma\) has either the same direction as the initial velocity of \(\tilde\alpha_{+}\coloneqq \Phi\circ \alpha_{+}\), or the same direction as the initial velocity of \(\tilde\alpha_{-}\coloneqq \Phi\circ \alpha_{-}\). In any case, the initial velocity of \(\tilde\gamma\) is horizontal:
by Proposition \ref{prop: opposite horizontal initial velocities (beta-isomorphism)}, the initial velocities of \(\tilde\alpha_{+}\) and \(\tilde\alpha_{-}\) are both horizontal (and have opposite directions) because \(\Phi\) is a 
\(\beta\)-isomorphism.

Now, we prove that \((ii)\implies (i)\). Assume that condition \((ii)\) is satisfied.
We show that \(\lim_{x\to 0^+}\Phi_1(x,0)/x\neq 0\) \ and 
\ \(\lim_{x\to 0^-}\Phi_1(x,0)/x\neq 0\).

Since the initial velocity of the path \(\alpha_{+}\) is horizontal, it follows from our assumption that the initial velocity of the path 
\(\tilde\alpha_{+} = \Phi\circ \alpha_{+}\) is horizontal. 
Hence, \(\lim_{x\to 0^+}\Phi_1(x,0)/x\neq 0\).
Similarly, since the initial velocity of the path \(\alpha_{-}\) is horizontal, it follows from our assumption that the initial velocity of the path 
\(\tilde\alpha_{-} = \Phi\circ \alpha_{-}\) is horizontal. 
Hence, \(\lim_{x\to 0^+}\Phi_1(-x,0)/x\neq 0\). And since
\(\lim_{x\to 0^-}\Phi_1(x,0)/x = - \lim_{x\to 0^+}\Phi_1(-x,0)/x\), we have
\(\lim_{x\to 0^-}\Phi_1(x,0)/x\neq 0\).
\end{proof}

Let \(\Phi\) and \(\Psi\) be \(\beta\)-isomorphisms of degree \(d\).  We say that \(\Psi\) is {\it composable} with \(\Phi\) if there exist 
\(\beta\)-quasihomogeneous polynomials \(F(X,Y),G(X,Y),H(X,Y)\) of the degree \(d\) such that \(H\circ\Psi = G\) \ and \ \(G\circ \Phi = F\).

\begin{proposition}
Let \(\Phi,\Psi\from (\R^2,0)\to (\R^2,0)\) be \(\beta\)-isomorphisms of degree 
\(d\). If \(\Psi\) is composable with \(\Phi\), then \(\Psi\circ\Phi\) is a 
\(\beta\)-isomorphism of degree \(d\).
\end{proposition}

\begin{proof}
Suppose that \(\Psi\) is composable with \(\Phi\). Then, there exist 
\(\beta\)-quasihomogeneous polynomials \(F(X,Y),G(X,Y),H(X,Y)\) of the same degree \(d\geq 1\) such that \(H\circ\Psi = G\) and \/\(G\circ \Phi = F\),
so \(H\circ(\Psi\circ\Phi) = F\). Now, take any continuous semialgebraic path 
\(\gamma\from [0,\epsilon)\to \R^2\) with horizontal initial velocity.
Since \(\Phi\) is a \(\beta\)-isomorphism, the initial velocity of the 
(continuous, semialgebraic) path \(\tilde\gamma\coloneqq \Phi\circ\gamma\)
is horizontal; and since \(\Psi\) is a \(\beta\)-isomorphism, the initial velocity of the (continuous, semialgebraic) path 
\((\Psi\circ\Phi)\circ\gamma = \Psi\circ{\tilde\gamma}\)
is horizontal. By Proposition \ref{prop: beta-isomorphism, horizontal initial velocity}, \(\Psi\circ\Phi\) is a \(\beta\)-isomorphism of degree \(d\).
\end{proof}

\begin{proposition}
The germ of the identity map \(I\from (\R^2,0)\to (\R^2,0)\) is a \(\beta\)-isomorphism of degree \(d\), for each \(d\geq 1\).
\end{proposition}

\begin{proof}
Fix any \(d\geq 1\). Clearly, there exist \(\beta\)-quasihomogeneous polynomials \(F(X,Y)\) and \/\(G(X,Y)\) of degree \(d\) such that 
\(G\circ I = F\): take any \(\beta\)-quasihomogeneous polynomial \(F(X,Y)\) of degree \(d\) (for example, \(F(X,Y) = X^d\)) and set \(G = F\). Also, for any continuous semialgebraic path \(\gamma\from [0,\epsilon)\to\R^2\) whose initial velocity is horizontal, the initial velocity of 
\(\tilde\gamma \coloneqq \Phi\circ\gamma\) is horizontal, because 
\(\tilde\gamma = \gamma\). By Proposition \ref{prop: beta-isomorphism, horizontal initial velocity}, \(I\) is \(\beta\)-isomorphism of degree \(d\).
\end{proof}

\begin{proposition}
Let \(\Phi\from (\R^2,0)\to(\R^2,0)\) be a germ of semialgebraic bi-Lipschitz map. If \(\Phi\) is a \(\beta\)-isomorphism of degree \(d\) then \(\Phi^{-1}\) is a 
\(\beta\)-isomorphism of degree \(d\).  
\end{proposition}

\begin{proof}
Since \(\Phi\) is a \(\beta\)-isomorphism of degree \(d\), there exist 
\(\beta\)-quasihomogeneous polynomials \(F(X,Y), G(X,Y)\) of degree \(d\) such that \(G\circ \Phi = F\). Hence, there exist \(\beta\)-quasihomogeneous polynomials \(\tilde F(X,Y), \tilde G(X,Y)\) of degree \(d\) such that \(\tilde G\circ\Phi^{-1} = \tilde F\): for example, 
take \(\tilde G = F\) and \/\(\tilde F = G\). 
Now, we show that for any continuous semialgebraic path 
\(\tilde\gamma\from[0,\epsilon)\to \R^2\) whose initial velocity is horizontal,
the initial velocity of \(\gamma\coloneqq \Phi^{-1}\circ{\tilde\gamma}\) is
horizontal. Fix one such path \(\tilde\gamma\from[0,\epsilon)\to \R^2\).
Let \(\alpha_+, \alpha_-\from[0,\epsilon)\to\R^2\) be the (continuous semialgebraic) paths defined by \(\alpha_+(t) = (t,0)\) and 
\/\(\alpha_-(t) = (-t,0)\). Since \(\Phi\) is a \(\beta\)-isomorphism, the initial velocities of the paths \(\tilde\alpha_+\coloneqq \Phi\circ\alpha_+\) and
\/\(\tilde\alpha_-\coloneqq \Phi\circ\alpha_-\) are both horizontal and have opposite directions. So, the initial velocity of \(\tilde\gamma\) (which is horizontal) has either the same direction as the initial velocity of 
\(\tilde\alpha_+\) or the same direction as the initial velocity of 
\(\tilde\alpha_-\). Since \(\Phi\) is bi-Lipschitz, it follows that the initial velocity of \(\gamma = \Phi^{-1}\circ{\tilde\gamma}\) has either the same direction as the initial velocity of \(\alpha_+ = \Phi^{-1}\circ{\tilde\alpha_+}\) or the same direction as the initial velocity of \(\alpha_- = \Phi^{-1}\circ{\tilde\alpha_-}\). In any case, the initial velocity of \(\gamma\) is horizontal. By Proposition \ref{prop: beta-isomorphism, horizontal initial velocity}, \(\Phi^{-1}\) is \(\beta\)-isomorphism of degree \(d\).
\end{proof}

\begin{remark}
Every \(\beta\)-isomorphism is composable with its inverse.
\end{remark}

Now, we proceed to the definition of the \(\beta\)-transform.
Given a \(\beta\)-isomorphism 
\(\Phi = (\Phi_1,\Phi_2)\from (\R^2,0) \to (\R^2,0)\),
we define 
\(\lambda \coloneqq (\lambda_+,\lambda_-)\) 
and 
\(\phi \coloneqq (\phi_+,\phi_-)\),
where:
\begin{align*}
\lambda_+ \coloneqq \lim_{x\to 0^+}\frac{\Phi_1(x,0)}{x}\,,& \quad
\lambda_- \coloneqq \lim_{x\to 0^-}\frac{\Phi_1(x,0)}{x}\\[7pt]
\phi_+(t) \coloneqq \abs{\lambda_+}^{-\beta}\cdot
          \lim_{x\to 0^+}\frac{\Phi_2(x,t\abs{x}^\beta)}{\abs{x}^\beta}\,,&\quad
\phi_-(t) \coloneqq \abs{\lambda_-}^{-\beta}\cdot
          \lim_{x\to 0^-}\frac{\Phi_2(x,t\abs{x}^\beta)}{\abs{x}^\beta}          
\end{align*}
We define the {\it\(\beta\)-transform} of \(\Phi\) to be the ordered pair 
\((\lambda,\phi)\). 

From the definitions above, it follows that, for each \(t\in\R\):
\begin{align}
\label{eq: Phi1 asymptotic formula}
\Phi_1(x,t\abs{x}^\beta) &= 
          \lambda x + o(x) \quad \text{ as } x\to 0\\
\label{eq: Phi2 asymptotic formula}
\Phi_2(x,t\abs{x}^\beta) &= 
          \abs{\lambda}^\beta\phi(t)\abs{x}^\beta + o(\abs{x}^\beta) 
          \quad \text{ as } x\to 0          
\end{align}
where
\begin{equation*}
\begin{cases}
\lambda = \lambda_+ \,\text{ and } \,\phi = \phi_+, &\text{if } \,x > 0\\
\lambda = \lambda_- \,\text{ and } \,\phi = \phi_-, &\text{if } \,x < 0
\end{cases}
\ .
\end{equation*}

\begin{proposition}
Let \(\Phi = (\Phi_1,\Phi_2)\) and \(\Psi = (\Psi_1,\Psi_2)\) be 
\(\beta\)-isomorphisms of the same degree such that \(\Psi\) is composable with \(\Phi\), and let \(Z \coloneqq \Psi\circ\Phi\).
Let \((\lambda,\phi)\), \((\mu,\psi)\) be the 
\(\beta\)-transforms of \(\Phi\), \(\Psi\), respectively. We have:
\begin{enumerate}[label=\roman*.]
\item Asympotic formula for \(Z_1(x,0)\).
\begin{equation}
\label{eq: asymptotic formula for Z1}
Z_1(x,0) = \lambda\mu x + o(x) \,\text{ as } x\to 0,
\end{equation}
where 
\begin{equation*}
\begin{cases}
\lambda = \lambda_+,\, \mu = \mu_+, & \text{if } \,x > 0\\
\lambda = \lambda_-,\, \mu = \mu_-, & \text{if } \,x < 0\\
\end{cases}
\quad\text{ or }\quad
\begin{cases}
\lambda = \lambda_+,\, \mu = \mu_-, & \text{if } \,x > 0\\
\lambda = \lambda_-,\, \mu = \mu_+, & \text{if } \,x < 0\\
\end{cases}
\ ,
\end{equation*}
according as \(\Phi\) is direct or reverse, respectively.

\item Asympotic formula for \(Z_2(x,t\abs{x}^\beta)\), with \(t\)  fixed. 
\begin{equation}
\label{eq: asymptotic formula for Z2}
Z_2(x,t\abs{x}^\beta) = 
\abs{\lambda\mu}^\beta\psi(\phi(t))\abs{x}^\beta + o(\abs{x}^\beta) 
\,\text{ as } x\to 0,
\end{equation} 
where
\begin{equation*}
\begin{cases}
\lambda = \lambda_+,\, \mu = \mu_+,\, \phi = \phi_+,\, \psi = \psi_+, 
& \text{if } \,x > 0\\
\lambda = \lambda_-,\, \mu = \mu_-,\, \phi = \phi_-,\, \psi = \psi_-, 
& \text{if } \,x < 0\\
\end{cases}
\end{equation*}
\quad\text{ or }\quad
\begin{equation*}
\begin{cases}
\lambda = \lambda_+,\, \mu = \mu_-,\, \phi = \phi_+,\, \psi = \psi_-, 
& \text{if } \,x > 0\\
\lambda = \lambda_-,\, \mu = \mu_+,\, \phi = \phi_-,\, \psi = \psi_+, 
& \text{if } \,x < 0\\
\end{cases}
\ ,
\end{equation*}
according as \(\Phi\) is direct or reverse, respectively.\end{enumerate}
\end{proposition} 

\begin{proof}
By applying formula (\ref{eq: Phi1 asymptotic formula}) 
successively to \(\Phi\) and \(\Psi\), we obtain:
\begin{align*}
Z_1(x,0) &= \Psi_1(\Phi(x,0))\\
    &= \Psi_1(\lambda x, \abs{\lambda}^\beta\phi(0)\abs{x}^\beta) + o(x)\\
    &= \lambda\mu x + o(x)
\end{align*}

Now, let \(t\in\R\) be fixed. By definition,
\begin{equation*}
Z_2(x,t\abs{x}^\beta) = \Psi_2(\tilde x, \Phi_2(x,t\abs{x}^\beta)),
\end{equation*}
where \(\tilde x \coloneqq \Phi_1(x,t\abs{x}^\beta)\).

Applying (\ref{eq: Phi2 asymptotic formula}) and 
(\ref{eq: Phi1 asymptotic formula}) successively, we obtain:
\begin{align*}
\Phi_2(x,t\abs{x}^\beta) &= 
          \abs{\lambda}^\beta\phi(t)\abs{x}^\beta + o(\abs{x}^\beta)\\
          &= \phi(t)\abs{\tilde x}^\beta + o(\abs{x}^\beta)
\end{align*}

Thus,
\begin{align*}
\Psi_2(\tilde x, \Phi_2(x,t\abs{x}^\beta)) &= 
     \Psi_2(\tilde x, \phi(t)\abs{\tilde x}^\beta) + o(\abs{x}^\beta)\\
     &= \abs{\mu}^\beta \psi(\phi(t))\abs{\tilde x}^\beta + o(\abs{x}^\beta)\\
     &= \abs{\lambda\mu}^\beta \psi(\phi(t))\abs{x}^\beta + o(\abs{x}^\beta)\,.
\end{align*}
\end{proof}

\begin{corollary}
Let \(\Phi = (\Phi_1,\Phi_2)\) and \(\Psi = (\Psi_1,\Psi_2)\) be 
\(\beta\)-isomorphisms, and let \(Z \coloneqq \Psi\circ\Phi\).
Let \((\lambda,\phi)\), \((\mu,\psi)\), and \((\nu,\zeta)\) be the 
\(\beta\)-transforms of \(\Phi\), \(\Psi\), and \(Z\), respectively. We have:
\begin{equation*}
(\nu,\zeta) = 
\begin{cases}
\left(
(\lambda_+\mu_+,\lambda_-\mu_-), (\psi_+\circ\phi_+,\psi_-\circ\phi_-)
\right), 
&\text{if \(\Phi\) is direct}\\
\left(
(\lambda_+\mu_-,\lambda_-\mu_+), (\psi_-\circ\phi_+,\psi_+\circ\phi_-)
\right), 
&\text{if \(\Phi\) is reverse}
\end{cases}
\end{equation*}
\end{corollary}

\begin{proof}
Dividing both sides of (\ref{eq: asymptotic formula for Z1}) by \(x\) and then successively letting \(x\to 0^+\) and \(x\to 0^-\), we obtain:
\begin{equation*}
\begin{cases}
\nu_+ = \lambda_+\mu_+ \,\text{ and }\, \nu_- = \lambda_-\mu_-, 
&\text{if \(\Phi\) is direct}\\
\nu_+ = \lambda_+\mu_- \,\text{ and }\, \nu_- = \lambda_-\mu_+, 
&\text{if \(\Phi\) is reverse}
\end{cases}
\end{equation*}

Thus, equation (\ref{eq: asymptotic formula for Z2}) can be rewritten as 
\begin{equation}
\label{eq: Asymptotic formula for Z2 rewritten}
Z_2(x,t\abs{x}^\beta) = 
\abs{\nu}^\beta\psi(\phi(t))\abs{x}^\beta + o(\abs{x}^\beta) 
\,\text{ as } x\to 0,
\end{equation} 
where
\begin{equation*}
\begin{cases}
\nu = \nu_+,\, \phi = \phi_+,\, \psi = \psi_+, 
& \text{if } \,x > 0\\
\nu = \nu_-,\, \phi = \phi_-,\, \psi = \psi_-, 
& \text{if } \,x < 0\\
\end{cases}
\quad\text{ or }\quad
\begin{cases}
\nu = \nu_+,\, \phi = \phi_+,\, \psi = \psi_-, 
& \text{if } \,x > 0\\
\nu = \nu_-,\, \phi = \phi_-,\, \psi = \psi_+, 
& \text{if } \,x < 0\\
\end{cases}
\ ,
\end{equation*}
according as \(\Phi\) is direct or reverse, respectively.

Dividing both sides of (\ref{eq: Asymptotic formula for Z2 rewritten}) by 
\(\abs{x}^\beta\) and then successively letting \(x\to 0^+\) and \(x\to 0^-\), 
we obtain:
\begin{equation*}
\begin{cases}
\zeta_+ = \psi_+\circ\phi_+ \,\text{ and }\, \zeta_- = \psi_-\circ\phi_-,
&\text{if \(\Phi\) is direct}\\
\zeta_+ = \psi_-\circ\phi_+ \,\text{ and }\, \zeta_- = \psi_+\circ\phi_-,
&\text{if \(\Phi\) is reverse}
\end{cases}
\end{equation*} 
\end{proof}

\begin{corollary}
\label{cor: beta-transform of the inverse of a beta-isomorphism}
Let \(\Phi\from(\R^2,0)\to(\R^2,0)\) be a \(\beta\)-isomorphism, 
and let \((\lambda,\phi)\) be the \(\beta\)-transform of \(\Phi\).
The \(\beta\)-transform of \(\Phi^{-1}\) is given by:
\begin{equation*}
\begin{cases}
\left((\lambda_+^{-1},\lambda_-^{-1}),(\phi_+^{-1},\phi_-^{-1})\right),
&\text{if\/ \(\Phi\) is direct}\\
\left((\lambda_-^{-1},\lambda_+^{-1}),(\phi_-^{-1},\phi_+^{-1})\right),
&\text{if\/ \(\Phi\) is reverse}
\end{cases}
\end{equation*} 
\end{corollary}

\begin{proposition}
Let \(\Phi\from(\R^2,0)\to(\R^2,0)\) be a \(\beta\)-isomorphism, and let 
\((\lambda,\phi)\) be the \(\beta\)-transform of\/ \(\Phi\). We have:
\begin{enumerate}[label=\textnormal{(\alph*)}]
\item \(\lambda_+\) and \/ \(\lambda_-\) are nonzero real numbers and they have the same sign.
\item \(\phi_+\) and \/ \(\phi_-\) are both bi-Lipschitz functions.
\end{enumerate}
\end{proposition}

\begin{proof}
By Remark \ref{rk: finite nonzero initial velocity}, the paths 
\(\Phi(x,0)\), \(0\leq x <\epsilon\), and \(\Phi(-x,0)\), \(0\leq x<\epsilon\), both have finite nonzero initial velocities, so 
\(\lambda_+ = \lim_{x\to 0^+}\Phi_1(x,0)/x\) and
\(\lambda_- = \lim_{x\to 0^-}\Phi_1(x,0)/x\)
are both nonzero real numbers. As pointed out just after the proof of Proposition \ref{prop: opposite horizontal initial velocities (beta-isomorphism)}, \(\lambda_+\) and \(\lambda_-\) have the same sign.

By (\ref{eq: Phi2 asymptotic formula}), 
for any fixed \(t\) and \(t^\prime\), we have:
\begin{equation*}
\phi_+(t) - \phi_+(t^\prime) =
 \abs{\lambda_+}^{-\beta}\cdot
 \frac{\Phi_2(x,t\abs{x}^\beta) - \Phi_2(x,t^\prime\abs{x}^\beta)}{\abs{x}^\beta}
+ \frac{o(\abs{x}^\beta)}{\abs{x}^\beta}, \quad\text{as } \,x\to 0^+
\end{equation*}
and
\begin{equation*}
\phi_-(t) - \phi_-(t^\prime) =
 \abs{\lambda_-}^{-\beta}\cdot
 \frac{\Phi_2(x,t\abs{x}^\beta) - \Phi_2(x,t^\prime\abs{x}^\beta)}{\abs{x}^\beta}
+ \frac{o(\abs{x}^\beta)}{\abs{x}^\beta}, \quad\text{as } \,x\to 0^-.
\end{equation*}

On the other hand, since \(\Phi_2\) is Lipschitz, there exists \(K > 0\) 
(independent of \(x,t,t^\prime\)) such that
\begin{equation*}
\abs{\Phi_2(x,t\abs{x}^\beta) - \Phi_2(x,t^\prime\abs{x}^\beta)} \leq
     K\abs{t - t^\prime}\abs{x}^\beta.
\end{equation*}

Thus,
\begin{equation*}
\abs{\phi_+(t) - \phi_+(t^\prime)} \leq 
\abs{\lambda_+}^{-\beta}K\abs{t - t^\prime} 
     + \frac{o(\abs{x}^\beta)}{\abs{x}^\beta} \quad\text{as } \,x\to 0^+
\end{equation*}
and
\begin{equation*}
\abs{\phi_-(t) - \phi_-(t^\prime)} \leq 
\abs{\lambda_-}^{-\beta}K\abs{t - t^\prime} 
     + \frac{o(\abs{x}^\beta)}{\abs{x}^\beta} \quad\text{as } \,x\to 0^-.
\end{equation*}

By successively letting \(x\to 0^+\) and \(x\to 0^-\), we obtain:
\begin{equation*}
\abs{\phi_+(t) - \phi_+(t^\prime)} \leq 
     \abs{\lambda_+}^{-\beta}K\abs{t - t^\prime} 
\quad\text{ and }\quad
\abs{\phi_-(t) - \phi_-(t^\prime)} \leq 
     \abs{\lambda_-}^{-\beta}K\abs{t - t^\prime} \,.
\end{equation*}
Therefore, \(\phi_+\) and \(\phi_-\) are both Lipschitz functions.

Up to this point, our argument shows that, for any \(\beta\)-isomorphism 
\(\Phi\), the functions \(\phi_+\) and \(\phi_-\) are both Lipschitz. By Corollary \ref{cor: beta-transform of the inverse of a beta-isomorphism}, this implies that \(\phi_+^{-1}\) and \(\phi_-^{-1}\) are both Lipschitz too.
\end{proof}

\section{\(\cal R\)-semialgebraic Lipschitz equivalence and height functions}
\label{section: R-semialgebraic Lipschitz equivalence and height functions}

Two real \(\beta\)-quasihomogeneous polynomials
\(F(X,Y)\) and \(G(X,Y)\) of degree \(d\geq 1\) are said to be
{\it \({\cal R}\)-semialgebraically Lipschitz equivalent} 
if there exists a germ of semialgebraic bi-Lipschitz homeomorphism 
\(\Phi\from(\R^2,0)\to(\R^2,0)\) such that \(G\circ\Phi = F\).

With each \(\beta\)-quasihomogeneous polynomial \(F(X,Y)\), we associate two polynomial functions: the {\it right height function} 
\(f_+\from\R\to\R\), given by \(f_+(t) \coloneqq F(1,t)\), 
and the the {\it left height function} 
\(f_-\from\R\to\R\), given by \(f_-(t) \coloneqq F(-1,t)\).

In this section, we show that if any two \(\beta\)-quasihomogeneous polynomials of a certain special type are \({\cal R}\)-semialgebraically Lipschitz equivalent, then their height functions can be arranged in pairs of Lipschitz equivalent functions (Corollary \ref{cor: equivalent polynomials, Lipschitz equivalent height functions}).

\begin{proposition}
\label{prop: action of the beta-transform of a beta-isomorphism}
Let \(F(X,Y)\) and \(G(X,Y)\) be \(\beta\)-quasihomogeneous polynomials of degree \(d\geq 1\). Suppose that \(G\circ\Phi = F\), for a certain 
\(\beta\)-isomorphism \(\Phi\from(\R^2,0)\to(\R^2,0)\). 
Let \(f_+,f_-\) be the height functions of \(F\), \(g_+,g_-\) the height functions of \(G\), and let \((\lambda,\phi)\) be the \(\beta\)-transform of \(\Phi\).
We have:
\begin{equation*}
\begin{cases}
g_+\circ \phi_+ = \abs{\lambda_+}^{-d} f_+ 
\quad\text{and\/}\quad
g_-\circ \phi_- = \abs{\lambda_-}^{-d} f_-,
&\text{if\/ \(\Phi\) is direct}\\[2pt]
g_-\circ \phi_+ = \abs{\lambda_+}^{-d} f_+
\quad\text{and\/}\quad
g_+\circ \phi_- = \abs{\lambda_-}^{-d} f_-, 
&\text{if\/ \(\Phi\) is reverse}
\end{cases}
\end{equation*}
\end{proposition}

\begin{proof}
By hypothesis, \(G(\Phi(x,y)) = F(x,y)\). 
Thus, for all \(t\in\R\) and \(x\neq 0\) sufficiently small, 
\begin{equation*}
G(\Phi(x,t\abs{x}^\beta)) = F(x,t\abs{x}^\beta)\,.
\end{equation*}
Since \(F\) is \(\beta\)-quasihomogeneous, this implies that
\begin{equation*}
G(\Phi(x,t\abs{x}^\beta)) = \abs{x}^d f(t),
\end{equation*}
where
\begin{equation*}
f = 
\begin{cases}
f_+, &\text{ if }x > 0\\
f_-, &\text{ if }x < 0
\end{cases}
\ .
\end{equation*}
Multiplying both sides of this equation by \(\abs{x}^{-d}\), we obtain:
\begin{equation*}
G\left(\frac{\Phi_1(x,t\abs{x}^\beta)}{\abs{x}}, 
    \frac{\Phi_2(x,t\abs{x}^\beta)}{\abs{x}^\beta}\right) = f(t)
\end{equation*}

Letting \(x\to 0^+\), it follows that
\begin{equation*}
G\left(\lambda_+, 
    \lim_{x\to 0^+}\frac{\Phi_2(x,t \abs{x}^\beta)}{\abs{x}^\beta}\right) = f_+(t)\,.
\end{equation*}

Hence,
\begin{equation}
\label{eq: g phi+}
g(\phi_+(t)) = \abs{\lambda_+}^{-d}f_+(t),
\end{equation}
where
\begin{equation*}
g = 
\begin{cases}
g_+, & \text{ if }\lambda_+ > 0\\
g_-, & \text{ if }\lambda_+ < 0
\end{cases}
\ .
\end{equation*}

Similarly, letting \(x\to 0^-\), it follows that
\begin{equation*}
G\left(-\lambda_-, 
    \lim_{x\to 0^-}\frac{\Phi_2(x,t \abs{x}^\beta)}{\abs{x}^\beta}\right) = f_-(t)\,.
\end{equation*}

Hence,
\begin{equation}
\label{eq: g phi-}
g(\phi_-(t)) = \abs{\lambda_-}^{-d}f_-(t),
\end{equation}
where
\begin{equation*}
g = 
\begin{cases}
g_-, & \text{ if }\lambda_- > 0\\
g_+, & \text{ if }\lambda_- < 0
\end{cases}
\ .
\end{equation*}

Clearly, equations (\ref{eq: g phi+}) and (\ref{eq: g phi-}) yield the result.
\end{proof}

\begin{theorem}
Let \(F(X,Y)\) and \(G(X,Y)\) be \(\beta\)-quasihomogeneous polynomials of degree \(d\geq 1\). Suppose that \(F\) and \(G\) are 
\({\cal R}\)-semialgebraically Lipschitz equivalent, and let  
\(\Phi\from(\R^2,0)\to(\R^2,0)\) be a germ of semialgebraic bi-Lipschitz homeomorphism such that \(G\circ\Phi = F\).

If \/
\(
V(F) \cap \{x > 0\} \neq \,\emptyset
\,\text{ and }
\,V(F) \cap \{x < 0\} \neq \,\emptyset,
\)
then \(\Phi\) is a \(\beta\)-isomorphism. Moreover, we have:
\begin{equation*}
\begin{cases}
g_+\circ \phi_+ = \abs{\lambda_+}^{-d} f_+ 
\quad\text{and\/}\quad
g_-\circ \phi_- = \abs{\lambda_-}^{-d} f_-,
&\text{if\/ \(\Phi\) is direct}\\[2pt]
g_+\circ \phi_- = \abs{\lambda_-}^{-d} f_- 
\quad\text{and\/}\quad
g_-\circ \phi_+ = \abs{\lambda_+}^{-d} f_+,
&\text{if\/ \(\Phi\) is reverse}
\end{cases},
\end{equation*}
where \(f_+,f_-\) are the height functions of \(F\), \(g_+,g_-\) are the height functions of \(G\), and \((\lambda,\phi)\) is the \(\beta\)-transform of \(\Phi\).
\end{theorem}

\begin{proof}
Suppose that 
\(V(F) \cap \{x > 0\} \neq \,\emptyset\)
and
\(V(F) \cap \{x < 0\} \neq \,\emptyset\).

Since 
\begin{equation*}
V(F) \cap \left\{x > 0\right\} = \bigcup_{t \in f_+^{-1}(0)} \left\{ (x,tx^\beta) : x > 0\right\},
\end{equation*}
the condition \(V(F) \cap \{x > 0\} \neq \,\emptyset\) implies that 
\,\(f_+^{-1}(0) \neq \emptyset\).

Take \(t_0\in\R\) such that \(f_+(t_0) = 0\).
In the notation used in the proof of Proposition \ref{prop: order of growth of the second coordinate of a beta-isomorphism}, we have:
\begin{equation*}
\abs{\frac{\Phi_2(x,t_0 \abs{x}^\beta)}{\abs{x}^\beta}} \leq
1 + \max\left\{ 
     \abs{\frac{c_{m-1}}{c_m}}\abs{\tilde x}^r, \ldots, 
     \abs{\frac{c_{1}}{c_m}}\abs{\tilde x}^{r(m-1)},
     \abs{\frac{c_0}{c_m}}\abs{\tilde x}^{rm}
     \right\} \,,
\end{equation*}
for \(x > 0\) sufficiently small. 
Since \(\lim_{x\to 0^+}\abs{\Phi_1(x,t_0\abs{x}^\beta)/x} < \infty\)
(this is guaranteed by Remark \ref{rk: finite nonzero initial velocity} along with equation (\ref{eq: equality of first components of initial velocities})), 
it follows that \(\Phi_2(x,t_0\abs{x}^\beta)/\abs{x}^\beta\) is bounded, 
for \(x > 0\) sufficiently small. 

Similarly, the assumption that \(V(F) \cap \{x < 0\} \neq \,\emptyset\) ensures the existence of \(t_1\in\R\) such that \(f_-(t_1) = 0\) and then , by adapting the argument above, we can prove that \(\Phi_2(x,t_1\abs{x}^\beta)/\abs{x}^\beta\) is bounded, for \(x < 0\) sufficiently small. 
By Proposition \ref{prop: order of growth of the second coordinate of a beta-isomorphism}, it follows that \(\Phi\) is a \(\beta\)-isomorphism. Now, the final statement is an immediate consequence of Proposition \ref{prop: action of the beta-transform of a beta-isomorphism}.
\end{proof}

\begin{corollary}
\label{cor: equivalent polynomials, Lipschitz equivalent height functions}
Let \(F(X,Y)\) and \(G(X,Y)\) be \(\beta\)-quasihomogeneous polynomials of degree \(d\geq 1\). Suppose that 
\(
V(F) \cap \{x > 0\} \neq \,\emptyset
\,\text{ and }
\,V(F) \cap \{x < 0\} \neq \,\emptyset.
\)
If \(F\) and \(G\) are \({\cal R}\)-semialgebraically Lipschitz equivalent, then
\begin{equation*}
\begin{cases}
f_+\cong g_+ \text{ and } \,f_-\cong g_-, & \text{ if \,\(\Phi\) is direct}\\ 
f_+\cong g_- \text{ and } \,f_-\cong g_+, & \text{ if \,\(\Phi\) is reverse}
\end{cases}
\ .
\end{equation*}
\end{corollary}

\section{The group of proto-transitions}
\label{section: proto-transitions}

Let \(\R^*\) be the multiplicative group of all nonzero real numbers, and 
let \({\cal L}\) be the group of all bi-Lipschitz Nash diffeomorphisms on \(\R\). 
Let \(H\coloneqq \{(\lambda_1,\lambda_2) \in \R^*\times\R^*:
\lambda_1\lambda_2 > 0\)\} (considered as a subgroup of the direct product 
\(\R^*\times\R^*\)), and let \(K\coloneqq {\cal L} \times {\cal L}\) (direct product). Define a binary operation on \(H\times K\) by setting:
\begin{equation*}
(\mu,\psi)\circ(\lambda,\phi)\coloneqq
\begin{cases}
\left( (\lambda_1\mu_1, \lambda_2\mu_2), 
    (\psi_1\circ \phi_1,\psi_2\circ\phi_2) \right),
    & \text{if \(\lambda > 0\)}\\
\left( (\lambda_1\mu_2, \lambda_2\mu_1), 
    (\psi_2\circ \phi_1,\psi_1\circ\phi_2) \right),
    & \text{if \(\lambda < 0\)}
\end{cases}
\end{equation*}
for all \((\lambda,\phi) = ((\lambda_1,\lambda_2),(\phi_1,\phi_2))\) and 
\((\mu,\psi) = ((\mu_1,\mu_2),(\psi_1,\psi_2))\),
where \(\lambda > 0\) means that \(\lambda_1 > 0\) and \(\lambda_2 > 0\),
and \(\lambda < 0\) means that \(\lambda_1 < 0\) and \(\lambda_2 < 0\).

\begin{proposition}
\label{prop: modified group operation}
\((H\times K, \circ)\) is a group. We call it the group of {\it proto-transitions}.
\end{proposition}
\begin{proof}
Let \((H\times K, \cdot)\) be the direct product of \(H\) and \(K\), so that
\begin{equation*}
(\mu,\psi)\cdot(\lambda,\phi) = 
\left( (\lambda_1\mu_1,\lambda_2\mu_2), (\psi_1\circ\phi_1,\psi_2\circ\phi_2) \right).
\end{equation*}

We express the operation \(\circ\) in terms of the operation \(\cdot\)\,, and then we use this expression to show that \((H\times K,\circ)\) is a group. 
Let \(\iota\from H\times K\to H\times K\) be the identity map and 
\(\tau\from H\times K\to H\times K\) be given by 
\(\tau((\lambda_1,\lambda_2),(\phi_1,\phi_2)) = 
((\lambda_2,\lambda_1),(\phi_2,\phi_1))\).

Define \(\theta\from H\to \text{Aut}(H\times K)\) by
\begin{equation*}
\theta(\lambda) \coloneqq
\begin{cases}
\iota, &\text{if }\lambda > 0\\
\tau, &\text{if }\lambda < 0 
\end{cases}
\end{equation*}
Clearly, \(\theta\) is a group homomorphism.

Let \(\pi\from H\times K\to H\) be the projection homomorphism. Then,
\(\alpha \coloneqq \theta\circ\pi\from H\times K\to \text{Aut}(H\times K)\)
is a group homomorphism such that:
\begin{enumerate}[label=\roman*.]
\item \(\alpha\circ\varphi = \alpha\) for all \(\varphi\in \text{Im}\,\alpha\)
\item \(\text{Im}\,\alpha\) is an abelian subgroup of \(\text{Aut}(H\times K)\)
\end{enumerate}
Also, we have:
\begin{equation*}
(\mu,\psi)\circ(\lambda,\phi) = 
   \left(\alpha(\lambda,\phi)(\mu,\psi)\right)\cdot(\lambda,\phi).
\end{equation*}
Hence, the result follows from the following lemma.
\begin{lemma}
\label{lemma: modified group operation}
Let \((G,\cdot)\) be a group and let \(\alpha\from G\to\text{Aut}(G)\) be a group homomorphism satisfying the following conditions: 
\begin{enumerate}[label = \roman*.]
\item \(\alpha\circ\varphi = \alpha\) for all \(\varphi\in \textnormal{Im}\,\alpha\)
\item \(\text{Im}\,\alpha\) is an abelian subgroup of \(\text{Aut}(G)\)
\end{enumerate}
Define a new operation \(\circ\) on \(G\) by setting
\(g\circ h \coloneqq \left(\alpha(h)(g)\right)\cdot h\).
Then, \((G,\circ)\) is a group.
\end{lemma}

\noindent
{\it Proof of Lemma \ref{lemma: modified group operation}.} First, we prove that the new operation is associative. For all \(g_1,g_2,g_3\in G\), we have:
\begin{align}
(g_1\circ g_2)\circ g_3 &= \left(\alpha(g_2)(g_1)\cdot g_2\right)\circ g_3 \nonumber \\
      &= \alpha(g_3)\left(\alpha(g_2)(g_1)\cdot g_2\right)\cdot g_3 \nonumber\\
      &= \alpha(g_3)\left(\alpha(g_2)(g_1)\right)\cdot\alpha(g_3)(g_2)\cdot g_3
      \label{Associativity of circ, part 1}
\end{align} 

On the other hand,
\begin{align}
g_1\circ (g_2\circ g_3) &= g_1\circ (\alpha(g_3)(g_2)\cdot g_3)
\nonumber\\
&= \alpha(\alpha(g_3)(g_2)\cdot g_3)(g_1)\cdot\alpha(g_3)(g_2)\cdot g_3
\nonumber\\
&= \alpha(\alpha(g_3)(g_2))(\alpha(g_3)(g_1))\cdot\alpha(g_3)(g_2)\cdot g_3
\nonumber\\
&\overset{(i)}{=} \alpha(g_2)(\alpha(g_3)(g_1))\cdot\alpha(g_3)(g_2)\cdot g_3
\nonumber\\
&\overset{(ii)}{=} \alpha(g_3)(\alpha(g_2)(g_1))\cdot\alpha(g_3)(g_2)\cdot g_3
\label{Associativity of circ, part 2}
\end{align}
From (\ref{Associativity of circ, part 1}) and (\ref{Associativity of circ, part 2}), it follows that \((g_1\circ g_2)\circ g_3 = g_1\circ (g_2\circ g_3)\).

Now we prove that the identity element \(1\) of the group \((G,\cdot)\) is also an identity element of \(G\) with respect to the operation \(\circ\). In fact, for all \(g\in G\), we have:
\begin{align*}
&1\circ g = \alpha(g)(1)\cdot g = 1\cdot g = g\\
&g\circ 1 = \alpha(1)(g)\cdot 1 = \id(g) = g
\end{align*}

Finally, we prove that each element \(g\in G\) has an inverse with respect to the operation \(\circ\). First, notice that for all \(g,h\in G\),
\begin{equation*}
h\circ g = 1 \iff \alpha(g)(h)\cdot g = 1 \iff \alpha(g)(h) = g^{-1}
\iff h = \alpha(g)^{-1}(g^{-1}) \iff h = \alpha(g^{-1})(g^{-1})\,.
\end{equation*}
Thus, \(h = \alpha(g^{-1})(g^{-1})\) is a left inverse of \(g\) with respect to the operation \(\circ\). Let us show that \(h\) is also a right inverse of \(g\). In fact,
\begin{align*}
g\circ\alpha(g^{-1})(g^{-1}) &= 
\alpha(\alpha(g^{-1})(g^{-1}))(g)\cdot \alpha(g^{-1})(g^{-1}) \overset{(i)}{=} 
\alpha(g^{-1})(g)\cdot \alpha(g^{-1})(g^{-1})\\
&= \alpha(g^{-1})(g\cdot g^{-1}) = \alpha(g^{-1})(1) = 1\,.
\end{align*}
Hence, \(h = \alpha(g^{-1})(g^{-1})\) is the inverse of \(g\) with respect to the operation \(\circ\).
\end{proof}

Now, we define a family of actions of the group of proto-transitions on the set 
\(C^\omega\times C^\omega\), where \(C^\omega\) is the set of all real analytic functions on \(\R\).

\begin{proposition}
For each \(d\geq 1\), the map \(\circ\from (C^\omega\times C^\omega)\times (H\times K) \to C^\omega\times C^\omega\) defined by
\begin{equation*}
(g_1,g_2)\circ (\lambda,\phi) \coloneqq
\begin{cases}
\left(\abs{\lambda_1}^d g_1\circ\phi_1, \abs{\lambda_2}^d g_2\circ\phi_2\right),
&{if}\quad\lambda > 0\\[5pt]
\left(\abs{\lambda_1}^d g_2\circ\phi_1, \abs{\lambda_2}^d g_1\circ\phi_2\right),
&{if}\quad\lambda < 0
\end{cases}
\end{equation*}
is an action of the group of proto-transitions on \(C^\omega\times C^\omega\).
\end{proposition}
\begin{proof}
First, we notice that the map
\(\cbullet\from (C^\omega\times C^\omega)\times (H\times K) \to C^\omega\times C^\omega\) given by
\begin{equation*}
(g_1,g_2)\cbullet (\lambda,\phi)\coloneqq
(\abs{\lambda_1}^d g_1\circ\phi_1, \abs{\lambda_2}^d g_2\circ\phi_2)
\end{equation*}
is an action of the direct product \((H\times K,\cdot)\) on 
\(C^\omega\times C^\omega\). In fact, we have:
\begin{align*}
((g_1,g_2)\cbullet (\mu,\psi)) \cbullet (\lambda,\phi) &= 
(\abs{\mu_1}^d g_1\circ\psi_1, \abs{\mu_2}^d g_2\circ\psi_2)\cbullet (\lambda,\phi)\\
&= (\abs{\lambda_1}^d\abs{\mu_1}^d(g_1\circ\psi_1)\circ\phi_1,
\abs{\lambda_2}^d\abs{\mu_2}^d(g_2\circ\psi_2)\circ\phi_2)\\
&= (\abs{\lambda_1\mu_1}^d g_1\circ(\psi_1\circ\phi_1), 
\abs{\lambda_2\mu_2}^d g_2\circ(\psi_2\circ\phi_2))\\
&= (g_1,g_2)\cbullet((\mu,\psi)\cdot(\lambda,\phi))
\end{align*}
and
\begin{equation*}
(g_1,g_2)\cbullet((\id_\R,\id_\R),(1,1))  
= (\abs{1}^d g_1\circ \id_\R, \abs{1}^d g_2\circ \id_\R)
= (g_1,g_2)
\end{equation*}

Now, we express the map \(\circ\) in terms of the action \(\cbullet\), and then we use this expression to show that the map \(\circ\) is an action of the group of proto-transitions on \(C^\omega\times C^\omega\).

Denote by \(\text{Bij}(C^\omega\times C^\omega)\) the group of all bijections on \(C^\omega\times C^\omega\). Let \(I\from C^\omega\times C^\omega\to C^\omega\times C^\omega\) be the identity map and \(T\from C^\omega\times C^\omega\to C^\omega\times C^\omega\) be given by 
\(T(g_1,g_2) = (g_2,g_1)\). Define 
\(\Theta\from H\to\text{Bij}(C^\omega\times C^\omega)\) by
\begin{equation*}
\Theta(\lambda) \coloneqq
\begin{cases}
I, &\text{if}\quad\lambda > 0\\
T, &\text{if}\quad\lambda < 0\\
\end{cases}\,.
\end{equation*}
Clearly, \(\Theta\) is a group homomorphism. Let \(\pi\from H\times K\to H\) be the projection homomorphism. Then 
\(A\coloneqq \Theta\circ\pi
     \from (H\times K,\circ) \to \text{Bij}(C^\omega\times C^\omega)\)
is a group homomorphism such that:
\begin{enumerate}[label=\Roman*.]
\item \(A(\mu,\psi)((g_1,g_2)\cbullet(\lambda,\phi)) = 
(A(\mu,\psi)(g_1,g_2))\cbullet\alpha(\mu,\psi)(\lambda,\phi)\),
where \(\alpha = \theta\circ\pi\from H\times K\to \text{Aut}(H\times K,\cdot)\) is the group homomorphism defined in the proof of Proposition \ref{prop: modified group operation}.
\item \(\text{Im}\,A\) is an abelian subgroup of 
\(\text{Bij}(C^\omega\times C^\omega)\).
\end{enumerate}
Also, we have:
\begin{equation*}
(g_1,g_2)\circ(\lambda,\phi) = 
     (A(\lambda,\phi)(g_1,g_2))\cbullet(\lambda,\phi).
\end{equation*}
Hence, the result follows from the following lemma.
\begin{lemma}
\label{lemma: modified group action}
Let us use the notation of the Lemma \ref{lemma: modified group operation}. Also, let \(X\) be a set, \(\cbullet\from X\times(G,\cdot)\to X\) a group action, and \(A\from (G,\circ)\to \text{Bij\/}(X)\) a group homomorphism satisfying the following conditions:
\begin{enumerate}[label=\Roman*.]
\item \(A(h)(x\cbullet g) = 
(A(h)(x))\cbullet\alpha(h)(g)\) for all \(x\in X\) and \(g,h\in G\).
\item \(\text{Im}\,A\) is an abelian subgroup of \(\text{Bij\/}(X)\).
\end{enumerate}
Then, the map \(\circ\from X\times(G,\circ)\to X\) defined by
\begin{equation*}
x\circ g \coloneqq (A(g)(x))\cbullet g
\end{equation*}
is a group action.
\end{lemma}

\noindent
{\it Proof of Lemma \ref{lemma: modified group action}.}
For all \(x\in X\) and \(g,h\in G\), we have:
\begin{align*}
(x\circ g)\circ h &= 
     (A(h)(x\circ g))\cbullet h = (A(h)((A(g)(x))\cbullet g))\cbullet h\\
&\overset{(I)}{=} ((A(h)(A(g)(x)))\cbullet\alpha(h)(g))\cbullet h
= ((A(h)\circ A(g))(x))\cbullet(\alpha(h)(g)\cdot h)\\
&\overset{(II)}{=} ((A(g)\circ A(h))(x))\cbullet(\alpha(h)(g)\cdot h)
= (A(g\circ h)(x))\cbullet(g\circ h)\\
&= x\circ(g\circ h)
\end{align*}
Also, for all \(x\in X\),
\begin{equation*}
x\circ 1 = (A(1)(x))\cbullet 1 = A(1)(x) = \id_X(x) = x.
\end{equation*}
\end{proof}

\section{\(\beta\)-transitions and the inverse \(\beta\)-transform}
\label{section: beta-transitions and the inverse beta-transform}

Denote by \({\cal P}\) the set of all real polynomial functions on \(\R\). A proto-transition \((\lambda,\phi)\) is said to be a {\it \(\beta\)-transition} if the following conditions are satisfied:
\begin{enumerate}[label = \roman*.]
\item There exist pairs of nonconstant polynomial functions 
\((f_1,f_2), (g_1,g_2)\in {\cal P}\times{\cal P}\) such that
\begin{equation*}
(g_1,g_2)\circ (\lambda,\phi) = (f_1,f_2)
\end{equation*}
\item \(\displaystyle 
   \abs{\lambda_1}^\beta\cdot\lim_{\abs{t}\to\infty} \frac{\phi_1(t)}{t} = 
\abs{\lambda_2}^\beta\cdot\lim_{\abs{t}\to\infty} \frac{\phi_2(t)}{t}\)
\end{enumerate}

\begin{remark}
The limits in \((ii)\) are well-defined. In fact, for any homeomorphism 
\(\phi\from\R\to\R\), if there exist nonconstant polynomial functions 
\(f,g\from\R\to\R\) such that \(g\circ\phi = f\), then the limit 
\(\lim_{\abs{t}\to+\infty} \phi(t)/t\) exists and is finite.
\end{remark}

\begin{remark}
The identity element \(((1,1),(\id_\R,\id_\R))\) of the group of the proto-transitions is a \(\beta\)-transition, which we call the {\it identity \(\beta\)-transition}, and the inverse (with respect to the operation \(\circ\)) of a \(\beta\)-transition is a \(\beta\)-transition.
\end{remark}

Given a \(\beta\)-transition \((\lambda,\phi)\), we define a map 
\(\Phi\from\R^2\to\R^2\) by setting:
\begin{itemize}
\item \(\Phi(x,t\abs{x}^\beta) \coloneqq 
\left(\lambda_1 x, \abs{\lambda_1}^\beta \phi_1(t) \abs{x}^\beta\right),\quad
\text{for } x > 0,\, t\in\R\)

\item \(\Phi(x,t\abs{x}^\beta) \coloneqq 
\left(\lambda_2 x, \abs{\lambda_2}^\beta \phi_2(t) \abs{x}^\beta\right),\quad
\text{for } x < 0,\, t\in\R\)

\item \(\displaystyle\Phi(0,y) \coloneqq 
\left(0,\abs{\lambda_1}^\beta\lim_{\abs{t}\to+\infty} \frac{\phi_1(t)}{t}\,y\right)
= 
\left(0,\abs{\lambda_2}^\beta\lim_{\abs{t}\to+\infty} \frac{\phi_2(t)}{t}\,y\right),
\quad\text{for all }y\in\R\)
\end{itemize}
The {\it inverse \(\beta\)-transform} of \((\lambda,\phi)\) is the germ 
\(\Phi\from(\R^2,0)\to(\R^2,0)\) determined by the map \(\Phi\).

\begin{lemma}
Let \(\phi,\psi\from\R\to\R\) be functions for which the limits \/
\(\lim_{\abs{t}\to+\infty}\phi(t)/t\) and \/ \(\lim_{\abs{t}\to+\infty}\psi(t)/t\) exist and are nonzero. Then,
\begin{equation*}
\lim_{\abs{t}\to+\infty} \frac{\psi(\phi(t))}{t} = 
\lim_{\abs{t}\to+\infty} \frac{\psi(t)}{t} 
     \cdot \lim_{\abs{t}\to+\infty} \frac{\phi(t)}{t}\,.
\end{equation*}
\end{lemma}

\begin{proof}
Since \(\lim_{\abs{t}\to+\infty} \abs{\phi(t)/t} > 0\), we have 
\(\lim_{\abs{t}\to+\infty}\abs{\phi(t)} = +\infty\). Then,
\begin{equation*}
\lim_{\abs{t}\to+\infty}\frac{\psi(t)}{t} = 
\lim_{\abs{t}\to+\infty}\frac{\psi(\phi(t))}{\phi(t)}\,.
\end{equation*}
Hence,
\begin{equation*}
\lim_{\abs{t}\to+\infty}\frac{\psi(\phi(t))}{t} = 
\lim_{\abs{t}\to+\infty}\frac{\psi(t)}{t}
\cdot
\lim_{\abs{t}\to+\infty}\frac{\phi(t)}{t}\,.
\end{equation*}
\end{proof}

\begin{proposition}
Let \((\lambda,\phi)\) and \((\mu,\psi)\) be \(\beta\)-transitions, 
let \(\Phi\) and \(\Psi\) be their respective inverse \(\beta\)-transforms, and let 
\(Z\) be the inverse \(\beta\)-transform of \((\mu,\psi)\circ(\lambda,\phi)\).
Then, \(Z = \Psi\circ\Phi\).
\end{proposition}

\begin{proof}
For all \(x > 0, t\in \R\),
\begin{align*}
\Psi(\Phi(x,t\abs{x}^\beta)) &= 
     \Psi\left(\lambda_1 x, \abs{\lambda_1}^\beta \phi_1(t) \abs{x}^\beta\right)\\
&=
\begin{cases}
\left(\lambda_1\mu_1 x, 
     \abs{\lambda_1\mu_1}^\beta \psi_1(\phi_1(t)) \abs{x}^\beta\right), &
     \text{if }\lambda > 0\\[5pt]
\left(\lambda_1\mu_2 x, 
     \abs{\lambda_1\mu_2}^\beta \psi_2(\phi_1(t)) \abs{x}^\beta\right), &
     \text{if }\lambda < 0
\end{cases}\\
&= Z(x,t\abs{x}^\beta)\,.
\end{align*}

For all \(x < 0, t\in \R\),
\begin{align*}
\Psi(\Phi(x,t\abs{x}^\beta)) &= 
     \Psi\left(\lambda_2 x, \abs{\lambda_2}^\beta \phi_2(t) \abs{x}^\beta\right)\\
&=
\begin{cases}
\left(\lambda_2\mu_2 x, 
     \abs{\lambda_2\mu_2}^\beta \psi_2(\phi_2(t)) \abs{x}^\beta\right), &
     \text{if }\lambda > 0\\[5pt]
\left(\lambda_2\mu_1 x, 
     \abs{\lambda_2\mu_1}^\beta \psi_1(\phi_2(t)) \abs{x}^\beta\right), &
     \text{if }\lambda < 0
\end{cases}\\
&= Z(x,t\abs{x}^\beta)\,.
\end{align*}

Now, we prove that \(Z(0,y) = \Psi(\Phi(0,y))\), for all \(y\in\R\). 
For \(\lambda > 0\), we have:
\begin{align*}
Z(0,y) &= \left(0, 
   \abs{\lambda_1\mu_1}^\beta\cdot
   \lim_{\abs{t}\to+\infty} \frac{\psi_1\circ\phi_1(t)}{t}\cdot y\right)\\
&= \left(0, 
   \abs{\lambda_1\mu_1}^\beta\cdot
   \lim_{\abs{t}\to+\infty}\frac{\psi_1(t)}{t}\cdot
   \lim_{\abs{t}\to+\infty}\frac{\phi_1(t)}{t}\cdot y\right)\\
&= \left(0, 
   \left(\abs{\mu_1}^\beta
   \lim_{\abs{t}\to+\infty}\frac{\psi_1(t)}{t}\right)\cdot
   \left(\abs{\lambda_1}^\beta 
   \lim_{\abs{t}\to+\infty}\frac{\phi_1(t)}{t}\right)\cdot
   y\right)\\
&= \Psi(\Phi(0,y)).
\end{align*}
And for \(\lambda < 0\), we have:
\begin{align*}
Z(0,y) &= \left(0, 
   \abs{\lambda_1\mu_2}^\beta\cdot
   \lim_{\abs{t}\to+\infty} \frac{\psi_2\circ\phi_1(t)}{t}\cdot y\right)\\
&= \left(0, 
   \abs{\lambda_1\mu_2}^\beta\cdot
   \lim_{\abs{t}\to+\infty}\frac{\psi_2(t)}{t}\cdot
   \lim_{\abs{t}\to+\infty}\frac{\phi_1(t)}{t}\cdot y\right)\\
&= \left(0, 
   \left(\abs{\mu_2}^\beta
   \lim_{\abs{t}\to+\infty}\frac{\psi_2(t)}{t}\right)\cdot
   \left(\abs{\lambda_1}^\beta 
   \lim_{\abs{t}\to+\infty}\frac{\phi_1(t)}{t}\right)\cdot
   y\right)\\
&= \Psi(\Phi(0,y)).
\end{align*}
\end{proof}

\begin{proposition}
The inverse \(\beta\)-transform of the identity \(\beta\)-transition is the germ of the identity map \(I\from (\R^2,0)\to(\R^2,0)\).
\end{proposition}
\begin{proof}
Immediate from the definition.
\end{proof}

\begin{corollary}
\label{cor: Inverse beta-transform of the inverse of a beta-transition}
Let \((\lambda,\phi)\) be a \(\beta\)-transition and \(\Phi\) its inverse \(\beta\)-transform. Then, the inverse \(\beta\)-transform of \((\lambda,\phi)^{-1}\) is 
\(\Phi^{-1}\).
\end{corollary}

Now, we prove that the inverse \(\beta\)-transform of a \(\beta\)-transition is a germ of bi-Lipschitz map.

\begin{lemma}
\label{lemma: Lipschitz on the strip H_delta, right half-plane}
Let \(\phi\from\R\to\R\) be a bi-Lipschitz function such that \(g\circ \phi = f\) for some nonconstant polynomial functions \(f,g\from\R\to\R\), and let \(\lambda\) be a nonzero real number. Then, the map 
\(\Phi\from {\cal H}\to\R^2\), defined on the right half-plane 
\({\cal H}\coloneqq \{(x,y)\in\R^2: x > 0\}\) by
\begin{equation*}
\Phi(x,tx^\beta) \coloneqq (\lambda x, \abs{\lambda}^\beta\phi(t)x^\beta)
\end{equation*}
for all \(x>0\) and \(t\in\R\), is Lipschitz on the strip 
\({\cal H}_\delta \coloneqq \{(x,y)\in\R^2 : 0 < x < \delta\}\), for each 
\(\delta > 0\).
\end{lemma}

\begin{proof}
Let \(\delta > 0\) be fixed. We prove that \(\Phi\) is Lipschitz on both the upper half-strip \({\cal H}_\delta\cap \{y > 0\}\) and the 
lower half-strip \({\cal H}_\delta\cap \{y < 0\}\). Let us see that this implies the result. Assuming this claim, and using the fact that \(\Phi\) is continuous, we see that there exists a constant \(C > 0\) such that
\begin{equation}
\label{eq: Lipschitz condition for Phi on a quadrant}
\abs{\Phi(x_1,y_1) - \Phi(x_2,y_2)} \leq C\abs{(x_1,y_1) - (x_2,y_2)},
\end{equation} 
whenever \((x_1,y_1)\) and \((x_2,y_2)\) both belong to 
\({\cal H}_\delta\cap\{y\geq 0\}\) or to \({\cal H}_\delta\cap\{y\leq 0\}\).
We show that (\ref{eq: Lipschitz condition for Phi on a quadrant}) still holds for \((x_1,y_1)\in{\cal H}_\delta\cap\{y\geq 0\}\) 
and \((x_2,y_2)\in{\cal H}_\delta\cap\{y\leq 0\}\). 
Let \((\bar{x},0)\) be the point at which the line segment whose endpoints are \((x_1,y_1)\) and \((x_2,y_2)\) intersects the \(x\)-axis. By our assumptions, we have:
\begin{equation*}
\abs{\Phi(x_1,y_1) - \Phi(\bar{x},0)} \leq C\abs{(x_1,y_1) - (\bar{x},0)}
\end{equation*} 
and
\begin{equation*}
\abs{\Phi(\bar{x},0) - \Phi(x_2,y_2) } \leq C\abs{(\bar{x},0) - (x_2,y_2)}.
\end{equation*} 
Hence,
\begin{align*}
\abs{\Phi(x_1,y_1) - \Phi(x_2,y_2)} &\leq 
\abs{\Phi(x_1,y_1) - \Phi(\bar{x},0)} + \abs{\Phi(\bar{x},0) - \Phi(x_2,y_2)}\\
&\leq C\left(\abs{(x_1,y_1) - (\bar{x},0)} + \abs{(\bar{x},0) - (x_2,y_2)}\right)\\
&=C\abs{(x_1,y_1) - (x_2,y_2)}, 
\end{align*}
where the last equality holds because the point \((\bar{x},0)\) lies in the segment whose endpoints are \((x_1,y_1)\) and \((x_2,y_2)\). Therefore, our initial claim implies that \(\Phi\) is Lipschitz on the strip 
\({\cal H}_\delta\).

In order to establish our initial claim, we first show that for each fixed pair of points \((x_1,t_1x_1^\beta)\) and \((x_2,t_2x_2^\beta)\), either both on 
\({\cal H}_\delta\cap\{y > 0\}\) or both on \({\cal H}_\delta\cap\{y < 0\}\),
with \(x_1\neq x_2\)\/ and\/ \(t_1\neq t_2\),
there exist \(\omega\) between \(x_1\) and \(x_2\), and \(\tau\) between \(t_1\) and \(t_2\) such that 
\begin{equation}
\label{eq: algebraic trick}
\phi(t_2)x_2^\beta - \phi(t_1)x_1^\beta = 
(\phi(\tau) - \tau \phi^\prime(\tau))\cdot\beta\omega^{\beta - 1}\cdot(x_2 - x_1) + \frac{\phi(t_2) - \phi(t_1)}{t_2 - t_1}\cdot(t_2x_2^\beta - t_1x_1^\beta)
\end{equation}
In fact, 
\begin{align}
\phi(t_2) x_2^\beta - \phi(t_1) x_1^\beta &= 
(\phi(t_1) + q\cdot (t_2 - t_1)) x_2^\beta - \phi(t_1) x_1^\beta,
\text{ where } q = \frac{\phi(t_2) - \phi(t_1)}{t_2 - t_1}\nonumber\\
&= \phi(t_1)\cdot(x_2^\beta - x_1^\beta) + q\cdot(t_2x_2^\beta - t_1x_2^\beta)\nonumber\\
&= \phi(t_1)\cdot(x_2^\beta - x_1^\beta) + q\cdot(t_2x_2^\beta - t_1x_1^\beta) + q\cdot t_1\cdot(x_1^\beta - x_2^\beta)\nonumber\\
&= (\phi(t_1) - q\cdot t_1)\cdot(x_2^\beta - x_1^\beta) + q\cdot(t_2x_2^\beta - t_1x_1^\beta)\nonumber\\
&= \frac{t_2\phi(t_1) - t_1\phi(t_2)}{t_2 - t_1}\cdot(x_2^\beta - x_1^\beta) + \frac{\phi(t_2) - \phi(t_1)}{t_2 - t_1}\cdot(t_2x_2^\beta - t_1x_1^\beta)
\label{eq: algebraic identity for the second coordinate of the difference}
\end{align}

Since the points \((x_1,t_1x_1^\beta)\) and \((x_2,t_2x_2^\beta)\) are either both on \({\cal H}_\delta\cap\{y > 0\}\) or both on \({\cal H}_\delta\cap\{y < 0\}\), the real numbers \(t_1,t_2\) are either both positive or both negative. Thus, by Pompeiu's Mean Value Theorem \cite[p.~1~--~2]{Drago}, there exists a real number \(\tau\) between \(t_1\) and \(t_2\) such that 
\begin{equation}
\label{eq: Pompeiu's Mean Value Theorem}
\frac{t_2\phi(t_1)-t_1\phi(t_2)}{t_2 - t_1} = 
\phi(\tau) - \tau\phi^\prime(\tau).
\end{equation}

Also, by Lagrange's Mean Value Theorem, there exists a real number 
\(\omega\) between \(x_1\) and \(x_2\) such that
\begin{equation}
\label{eq: Lagrange's Mean Value Theorem}
x_2^\beta - x_1^\beta =\beta\omega^{\beta-1}\cdot (x_2 - x_1).
\end{equation}
Substituting (\ref{eq: Pompeiu's Mean Value Theorem}) and (\ref{eq: Lagrange's Mean Value Theorem}) in (\ref{eq: algebraic identity for the second coordinate of the difference}), we obtain (\ref{eq: algebraic trick}).

Now, since \(\phi\) is Lipschitz, there exists a constant \(C_1 > 0\) such that
\begin{equation*}
\abs{\frac{\phi(t_2) - \phi(t_1)}{t_2 - t_1}} \leq C_1,
\end{equation*}
for \(t_1\neq t_2\). On the other hand, since \(\phi(t)-t\phi^\prime(t)\) is bounded (see next section), there exists a constant \(C_2 > 0\) such that
\begin{equation*}
\abs{(\tau\phi^\prime(\tau) - \phi(\tau))\cdot\beta\omega^{\beta - 1}}
\leq C_2,
\end{equation*}
provided that \(0 < x_1,x_2 < \delta\).

Applying these bounds to (\ref{eq: algebraic trick}), we obtain:
\begin{equation*}
\abs{\phi(t_2) x_2^\beta - \phi(t_1) x_1^\beta} \leq 
C\cdot\left(\abs{x_2 - x_1} + \abs{t_2x_2^\beta - t_1x_1^\beta}\right),
\end{equation*}
where \(C = \max\{C_1,C_2\}\), for any pair of points \((x_1,t_1x_1^\beta)\) and \((x_2,t_2x_2^\beta)\), either both on \({\cal H}_\delta\cap{\{y > 0\}}\) or both on \({\cal H}_\delta\cap\{y < 0\}\), thereby proving our initial claim.
\end{proof}

\begin{corollary}
\label{cor: Lipschitz on the strip |x| < delta}
Let \((\lambda,\phi)\) be a \(\beta\)-transition, and let 
\(\Phi\from\R^2\to\R^2\) be the map defined by:
\begin{itemize}
\item \(\Phi(x,t\abs{x}^\beta) \coloneqq 
\left(\lambda_1 x, \abs{\lambda_1}^\beta \phi_1(t) \abs{x}^\beta\right),\quad
\text{for } x > 0,\, t\in\R\)

\item \(\Phi(x,t\abs{x}^\beta) \coloneqq 
\left(\lambda_2 x, \abs{\lambda_2}^\beta \phi_2(t) \abs{x}^\beta\right),\quad
\text{for } x < 0,\, t\in\R\)

\item \(\displaystyle\Phi(0,y) \coloneqq 
\left(0,\abs{\lambda_1}^\beta\lim_{\abs{t}\to+\infty} \frac{\phi_1(t)}{t}\,y\right)
= 
\left(0,\abs{\lambda_2}^\beta\lim_{\abs{t}\to+\infty} \frac{\phi_2(t)}{t}\,y\right),
\quad\text{for all }y\in\R\)
\end{itemize}
Then, there exist \(\delta > 0\) such that \(\Phi\) is Lipschitz on the strip
\(\{(x,y)\in\R^2 : \abs{x} < \delta\}\).
\end{corollary}

\begin{proof}
By Lemma \ref{lemma: Lipschitz on the strip H_delta, right half-plane}, there exist \(C_+ > 0\) and \(\delta_+ > 0\) such that 
\(\Phi\vert_{{\cal H}_{\delta_+}}\from{\cal H}_{\delta_+}\to\R^2\) is \(C_+\)-Lipschitz. Since \(\Phi\vert_{{\cal H}_{\delta_+}}\) is uniformly continuous and takes values in \(\R^2\) (which is a complete metric space), it has a unique continuous extension \(\widetilde\Phi\) to 
\(\widetilde{{\cal H}}_{\delta_+}\coloneqq \{(x,y)\in\R^2 : 0\leq x < \delta_+\}\). Let us show that \(\widetilde\Phi = \Phi\vert_{\widetilde{{\cal H}}_{\delta_+}}\). Obviously, \(\widetilde{\Phi}(x,y) = \Phi(x,y)\) for all 
\((x,y)\in{\cal H}_{\delta_+}\). And for all \(y\in\R\), we have:
\begin{align*}
\widetilde{\Phi}(0,y) &= \lim_{x\to 0^+}\Phi(x,y)\\
&= \lim_{x\to 0^+}\left(\lambda_1 x, 
\abs{\lambda_1}^\beta\cdot\frac{\phi_1(t)}{t}\cdot y\right), 
\quad\text{where } t = \frac{y}{x^\beta}\\
&= \left(0,
   \abs{\lambda_1}^\beta\cdot\lim_{\abs{t}\to+\infty}\frac{\phi_1(t)}{t}\cdot y\right)\\
&= \Phi(0,y).
\end{align*}
Hence, \(\Phi\vert_{\widetilde{\cal H}_{\delta_+}}\) is the continuous extension of \(\Phi\vert_{{\cal H}_{\delta_+}}\) to \(\widetilde{\cal H}_{\delta_+}\). And since \(\Phi\vert_{{\cal H}_{\delta_+}}\) is \(C_+\)-Lipschitz, it follows that 
\(\Phi\vert_{\widetilde{\cal H}_{\delta_+}}\) is \(C_+\)-Lipschitz too.

Similarly, we can prove that there exist \(C_- > 0\) and \(\delta_- > 0\) such that \(\Phi\vert_{-\widetilde{\cal H}_{\delta_-}}\from 
-\widetilde{\cal H}_{\delta_-}\to\R^2\) is \(C_-\)-Lipschitz. Therefore, 
\(\Phi\) is \(C\)-Lipschitz on the strip \(\{(x,y)\in\R^2 : \abs{x} < \delta\}\), where \(C = \max\{C_+,C_-\}\) and \(\delta = \min\{\delta_+,\delta_-\}\).
\end{proof}

\begin{corollary}
The inverse \(\beta\)-transform \(\Phi\from(\R^2,0)\to(\R^2,0)\) of every \(\beta\)-transition \((\lambda,\phi)\) is a germ of semialgebraic bi-Lipschitz map.
\end{corollary}

\begin{proof}
Since \(\phi_1\) and \(\phi_2\) are both semialgebraic functions, it is immediate from the definition of the inverse \(\beta\)-transform that \(\Phi\) is a germ of semialgebraic map. And by Corollary \ref{cor: Lipschitz on the strip |x| < delta}, both \(\Phi\) and \(\Phi^{-1}\) (which is the inverse 
\(\beta\)-transform of \((\lambda,\phi)^{-1}\), by Corollary \ref{cor: Inverse beta-transform of the inverse of a beta-transition}) are germs of Lipschitz maps. Hence the result.
\end{proof}

\begin{proposition}
Let \(F, G\in\R[X,Y]\) be \(\beta\)-quasihomogeneous polynomials of degree \(d\geq 1\), and let \(f_+,f_-\) be the height functions of \(F\) and \(g_+,g_-\) the height functions of \(G\). Suppose that 
\((g_+,g_-)\circ(\lambda,\phi) = (f_+,f_-)\)
for some \(\beta\)-transition \((\lambda,\phi)\).
Then, \(G\circ \Phi = F\), where \(\Phi\) is the inverse \(\beta\)-transform of \((\lambda,\phi)\).
\end{proposition}

\begin{proof}
First, we show that
\begin{equation}
\label{eq: preliminar relation between F and G}
G(\lambda_1, \abs{\lambda_1}^\beta\phi_1(t)) = F(1,t) 
\quad\text{and}\quad
G(-\lambda_2, \abs{\lambda_2}^\beta\phi_2(t)) = F(-1,t).
\end{equation}
We consider separately the cases \(\lambda > 0\) and \(\lambda < 0\).

If \(\lambda > 0\), we have
\begin{equation*}
\abs{\lambda_1}^d\cdot g_+\circ\phi_1 = f_+ 
\quad\text{and}\quad
\abs{\lambda_2}^d\cdot g_-\circ\phi_2 = f_-.
\end{equation*}
Equivalently,
\begin{equation*}
\abs{\lambda_1}^d\cdot G(1, \phi_1(t)) = F(1,t) 
\quad\text{and}\quad
\abs{\lambda_2}^d\cdot G(-1,\phi_2(t)) = F(-1,t).
\end{equation*}
Thus, using the fact that the polynomials \(F\) and \(G\) are \(\beta\)-quasihomogeneous of degree \(d\), we obtain (\ref{eq: preliminar relation between F and G}).

If \(\lambda < 0\), we have
\begin{equation*}
\abs{\lambda_1}^d\cdot g_-\circ\phi_1 = f_+ 
\quad\text{and}\quad
\abs{\lambda_2}^d\cdot g_+\circ\phi_2 = f_-.
\end{equation*}
Equivalently,
\begin{equation*}
\abs{\lambda_1}^d\cdot G(-1, \phi_1(t)) = F(1,t) 
\quad\text{and}\quad
\abs{\lambda_2}^d\cdot G(1,\phi_2(t)) = F(-1,t).
\end{equation*}
Again, using the fact that the polynomials \(F\) and \(G\) are \(\beta\)-quasihomogeneous of degree \(d\), we obtain (\ref{eq: preliminar relation between F and G}).

Now, by using once more the fact that the polynomials \(F\) and \(G\) are \(\beta\)-quasihomogeneous of degree \(d\), we obtain from (\ref{eq: preliminar relation between F and G}):
\begin{equation*}
G(\lambda_1 x, \abs{\lambda_1}^\beta\phi_1(t)\abs{x}^\beta) = 
F(x,t\abs{x}^\beta)\quad\text{for }x > 0,t\in\R 
\end{equation*}
and
\begin{equation*}
G(\lambda_2 x,\abs{\lambda_2}^\beta\phi_2(t)\abs{x}^\beta) = 
F(x,t\abs{x}^\beta)\quad\text{for }x < 0,t\in\R. 
\end{equation*}
In other words,
\begin{equation*}
G(\Phi(x,y)) = F(x,y)
\quad\text{for all }(x,y)\in\R^2, \text{ with }x\neq 0.
\end{equation*}
Since \(\Phi\) is continuous\footnote{Clearly, \(\Phi\) is continuous on the right half-plane \(\{(x,y)\in\R^2: x > 0\}\) and also on the left half-plane \(\{(x,y)\in\R^2: x < 0\}\). By Corollary \ref{cor: Lipschitz on the strip |x| < delta}, \(\Phi\) is Lipschitz (and therefore continuous) on a strip \(\{(x,y)\in\R^2:\abs{x}<\delta\}\). Since the right half-plane, the left half-plane, and the strip around the \(y\)-axis form an open cover of the plane, it follows that \(\Phi\) is continuous.}, we have
\(G(\Phi(x,y)) = F(x,y)\) for all \((x,y)\in\R^2\). 
\end{proof}

\section{The boundedness of \(\phi(t) - t\phi^\prime(t)\)}

Let \(\phi\from\R\to\R\) be a bi-Lipschitz function such that \(g\circ\phi = f\) for some nonconstant polynomials \(f,g\from\R\to\R\). We already know that \(\phi\) is bi-analytic (Lemma \ref{lemma: bi-Lipschitz iff bi-analytic}). In this section, we prove that the function \(\phi(t) - t\phi^\prime(t)\) is bounded.\footnote{The approach taken here was suggested by Prof. Maria Michalska.}

Consider the function \(\widehat{\phi}\from \R^{*}\cup\{\infty\}\to\R\) defined by
\begin{equation*}
\widehat{\phi}(t)\coloneqq
\begin{cases}
\phi(t)/t,&\text{if }t\in\R^*\\
\lim_{\abs{t}\to+\infty}\phi(t)/t,&\text{if }t=\infty
\end{cases}
\end{equation*}
By Lemma \ref{lemma: immediate consequences Lipschitz equivalence equation}, \(\widehat{\phi}\) is a well-defined continuous function. We prove that \(\widehat{\phi}\) is analytic at \(\infty\). As a consequence, we obtain the boundedness of \(\phi(t) - t\phi^\prime(t)\).

In order to do this, we consider the coordinate representation of 
\(\widehat{\phi}\) on the chart centered at \(\infty\), which is the function 
\(\psi\from\R\to\R\) given by
\begin{equation*}
\psi(t)\coloneqq
\begin{cases}
t\phi(t^{-1}), &\text{if }t\in\R\setminus\{0\}\\
\lim_{\abs{t}\to+\infty}\phi(t)/t,&\text{if } t =0
\end{cases}.
\end{equation*}

\noindent
{\bf Claim.} {\it \(\psi\) is analytic.}
\begin{proof}
Clearly, \(\psi\) is analytic at every point \(t\in\R\setminus\{0\}\), so we focus on proving the analiticity of \(\psi\) at \(t = 0\). 
Let \(P(X,Y)\coloneqq g(Y) - f(X)\) and let \(P^*(X,Y,Z)\) be the homogeneization of \(P\). 
Also, let \(f(t) = \sum_{i=0}^d a_it^i\) and \(g(t) = \sum_{i=0}^d b_it^i\), where 
\(a_d,b_d\neq 0\) and \(d\geq 1\), so
\begin{equation*}
P(X,Y) = \sum_{i=0}^d b_iY^i - \sum_{i=0}^d a_iX^i
\end{equation*}
and
\begin{equation*}
P^*(X,Y,Z) = \sum_{i=0}^d b_iY^iZ^{d-i} - \sum_{i=0}^d a_iX^iZ^{d-i}.
\end{equation*}

Since \(g(\phi(t)) = f(t)\), we have \(P(t,\phi(t)) = 0\) for all \(t\in\R\).
Equivalently, \({P^*(t,\phi(t),1) = 0}\) for all \(t\in\R\).
Since \(P^*\) is homogeneous, it follows that 
\(P^*(1,\phi(t)/t, 1/t) = 0\) for all \(t\in\R\setminus\{0\}\).
Equivalently,
\begin{equation*}
P^*(1,t\phi(t^{-1}), t) = 0 \quad\text{for all } t\in\R\setminus\{0\}.
\end{equation*}

Let \(\widetilde{P}(Y,Z)\coloneqq P^*(1,Y,Z)\). From the computations above, it follows that
\begin{equation*}
\widetilde{P}(\psi(t),t) = 0\quad\text{for all }t\in\R.
\end{equation*}
Now we prove that the equation \(\widetilde{P}(y,z) = 0\) determines \(y\) as an analytic function of \(z\) in a neighborhood of \((l,0)\), 
where \(l\coloneqq \lim_{\abs{t}\to+\infty}\phi(t)/t\). This will give the desired conclusion.

The partial derivatives of \(\widetilde{P}\) at \((y,0)\) are given by
\begin{equation*}
\frac{\partial\widetilde{P}}{\partial y}(y,0) = d\cdot b_d\cdot y^{d-1}
\quad\text{and}\quad
\frac{\partial\widetilde{P}}{\partial z}(y,0) = b_{d-1}\cdot y^{d-1} - a_{d-1}.
\end{equation*}
Since \(\widetilde{P}(l,0) = 0\) and 
\(\frac{\partial\widetilde{P}}{\partial y}(l,0) = d\cdot b_d\cdot l^{d-1}\neq 0\)
(by Lemma \ref{lemma: immediate consequences Lipschitz equivalence equation}, we have \(l\neq 0\)), there is an analytic function 
\(\widetilde{\psi}\from I\to J\) from an open interval \(I\) containing \(0\) to an open interval \(J\) containing \(l\) such that 
\begin{equation*}
\widetilde{P}(y,z) = 0 \iff y = \widetilde{\psi}(z),
\quad\text{for all }y\in J, z\in I.
\end{equation*}
Since \(\psi\) is continuous and \(\psi(0) = l\), there exists an open interval \(I_0\subseteq I\) containing \(0\) such that \(\psi(t)\in J\) for all \(t\in I_0\). And since \(\widetilde{P}(\psi(t),t) = 0\) for all \(t\in I_0\), it follows that 
\(\psi(t) = \widetilde{\psi}(t)\) for all \(t\in I_0\). Hence, \(\psi\) is analytic at \(t = 0\).
\end{proof}

Finally, we can prove that the function \(\phi(t) - t\phi^\prime(t)\) is bounded.
From the definition of \(\psi\), we see that
\begin{equation*}
\frac{\phi(t)}{t} = \psi(t^{-1})\quad\text{for all }t\in\R\setminus\{0\}.
\end{equation*}
Differentiating both sides of this equation, we get
\begin{equation*}
\frac{\phi^\prime(t)\cdot t - \phi(t)}{t^2} = -\frac{\psi^\prime(t^{-1})}{t^2}
\quad\text{for all }t\in\R\setminus\{0\}.
\end{equation*}
Equivalently, we have
\begin{equation*}
\phi(t) - t\phi^\prime(t) = \psi^\prime(t^{-1})
\quad\text{for all }t\in\R\setminus\{0\}.
\end{equation*}
Hence,
\begin{equation*}
\lim_{\abs{t}\to+\infty} \phi(t) - t\phi^\prime(t) = \psi^\prime(0).
\end{equation*}
Since the function \(\phi(t) - t\phi^{\prime}(t)\) is continuous on \(\R\), the existence of this limit implies its boundedness.

\section{A characterization of \(\beta\)-transitions}
\label{section: A characterization of beta-transitions}

Let \(F,G\in\R[X,Y]\) be \(\beta\)-quasihomogeneous polynomials of degree 
\(d\geq 1\), and let \(f_+,f_-\) be the height functions of \(F\) and \(g_+,g_-\) the height functions of \(G\). Suppose that 
\((g_+,g_-)\circ(\lambda,\phi) = (f_+,f_-)\)
for some proto-transition \((\lambda,\phi)\). We are interested in determining
when \((\lambda,\phi)\) can be chosen among the \(\beta\)-transitions.

Denote by \(e_F\) the multiplicity of \(X\) as a factor of \(F\) (in \(\R[X,Y]\)), and by \(e_G\) the multiplicity of \(X\) as a factor of \(G\). Then
\(F(X,Y) = X^{e_F}\cdot\widetilde{F}(X,Y)\) and 
\(G(X,Y) = X^{e_G}\cdot\widetilde{G}(X,Y)\), where
\(X\nmid\widetilde{F}(X,Y)\) and \(X\nmid\widetilde{G}(X,Y)\).

\begin{claim}
\label{claim: e_F = e_G}
\(e_F = e_G\)
\end{claim}

\begin{proof}
Let \(\beta = r/s\), where \(r > s > 0\) and \(\gcd(r,s) = 1\). Since \(F\) and \(G\) are \(\beta\)-quasihomogeneous polynomials of degree \(d\), we have:
\begin{equation*}
F(X,Y) = \sum_{k = 0}^m a_k X^{d - rk}Y^{sk}
\quad\text{and}\quad
G(X,Y) = \sum_{k = 0}^n b_k X^{d - rk}Y^{sk},
\end{equation*}
where \(a_m\neq 0\), \(b_n\neq 0\), and \(0\leq m,n\leq \lfloor d/r\rfloor\).
Then, \(e_F = d-rm\) and \(e_G = d-rn\).

We prove that \(m = n\); by the equations above, this implies that \(e_F = e_G\). Since \(f_+(t) = \sum_{k = 0}^m a_k t^{sk}\) and  
\(f_-(t) = \sum_{k = 0}^m (-1)^{d-rk}a_k t^{sk}\), we have 
\(\deg f_+ = \deg f_- = sm\); and since \(g_+(t) = \sum_{k = 0}^n b_k t^{sk}\) and  \(g_-(t) = \sum_{k = 0}^n (-1)^{d-rk}b_k t^{sk}\), we have 
\(\deg g_+ = \deg g_- = sn\). Thus, since \(f_+\) and \(f_-\) are Lipschitz equivalent to \(g_+\) and \(g_-\) in some order 
(by assumption, \((g_+,g_-)\circ(\lambda,\phi) = (f_+,f_-)\) for some proto-transition \((\lambda,\phi)\)), it follows that \(sm = sn\). Therefore, \(m = n\).
\end{proof}

From now on, we drop the subscript and denote simply by \(e\) the multiplicity  of \(X\) as a factor of either \(F\) or \(G\).

\begin{claim} 
\label{claim: F tilde and G tilde beta-quasihomogeneous of degree d-e}
For all \(t > 0\), we have:
\begin{equation*}
\widetilde{F}(tX,t^\beta Y) = t^{d-e}\widetilde{F}(X,Y)
\quad\text{and}\quad
\widetilde{G}(tX,t^\beta Y) = t^{d-e}\widetilde{G}(X,Y)
\end{equation*}
\end{claim}

\begin{proof}
Since \(F\) is a \(\beta\)-quasihomogeneous polynomial of degree \(d\),
\begin{align*}
F(tX,t^\beta Y) &= t^d F(X,Y)\\
&= t^d X^e \widetilde{F}(X,Y).
\end{align*}
On the other hand, since the multiplicity of \(X\) as a factor of \(F\) is equal to \(e\), 
\begin{align*}
F(tX,t^\beta Y) &= (tX)^e \widetilde{F}(tX, t^\beta Y)\\
&= t^e X^e \widetilde{F}(tX,t^\beta Y).
\end{align*}
Hence,
\begin{equation*}
t^e X^e \widetilde{F}(tX,t^\beta Y) = t^d X^e \widetilde{F}(X,Y).
\end{equation*}
Therefore,
\begin{equation*}
\widetilde{F}(tX,t^\beta Y) = t^{d - e} \widetilde{F}(X,Y).
\end{equation*}

Obviously, the deduction above with \(F\) replaced by \(G\) yields the other equation.
\end{proof}

Let us briefly consider the case where \(e = d\). In this case, we have 
\(F = aX^d\) and \(G = bX^d\), with \(a,b\neq 0\). 
The next proposition tells us how to determine whether any two such polynomials are \({\cal R}\)-semialgebraically Lipschitz equivalent.

\begin{proposition}
Let \(F(X,Y) = aX^d\) and \(G(X,Y) = bX^d\), where \(a,b\in\R\setminus\{0\}\) and \(d\geq 1\).
\begin{enumerate}[label = \roman*.]
\item If \(d\) is even, then \(F\) and \(G\) are \({\cal R}\)-semialgebraically Lipschitz equivalent if and only if \(a\) and \(b\) have the same sign.
\item If \(d\) is odd, then \(F\) and \(G\) are \({\cal R}\)-semialgebraically Lipschitz equivalent.
\end{enumerate}
\end{proposition}

\begin{proof}
(i) Suppose that \(d\) is even. If there exists a germ of semialgebraic bi-Lipschitz homeomorphism \(\Phi\from(\R^2,0)\to(\R^2,0)\) such that 
\(G\circ\Phi = F\) then \(b\cdot\Phi_1(x,y)^d = ax^d\) in a neighborhood of the origin, which implies that \(a\) and \(b\) have the same sign, since \(d\) is even. Now, assuming that \(a\) and \(b\) have the same sign, we have 
\(G\circ\Phi = F\), \(\Phi(x,y) = \left(\abs{\frac{a}{b}}^{\frac{1}{d}}\cdot x, y\right)\).

(ii) If \(d\) is odd, then \(G\circ\Phi = F\), 
\(\Phi(x,y) = \left(\left(\frac{a}{b}\right)^{\frac{1}{d}}\cdot x, y\right)\).
\end{proof}

Since our goal is ultimately to determine whether any two given \(\beta\)-quasihomogeneous polynomials are \({\cal R}\)-semialgebraically Lipschitz equivalent, from now on we focus on the case where \(e < d\).

\begin{claim}
\label{claim: limit condition, for e < d}
If \(e < d\), then 
\begin{equation*}
\abs{\lambda_1}^{\frac{d\beta}{d - e}}\cdot
   \lim_{\abs{t}\to+\infty} \abs{\frac{\phi_1(t)}{t}} = 
\abs{\lambda_2}^{\frac{d\beta}{d - e}}\cdot
   \lim_{\abs{t}\to+\infty} \abs{\frac{\phi_2(t)}{t}}.
\end{equation*}
\end{claim}

\begin{proof}
Throughout this proof, we assume that \(t > 0\). First, we consider the case where \(\lambda > 0\). In this case, we have
\begin{equation*}
\abs{\lambda_1}^d\cdot g_+\circ\phi_1 = f_+
\quad\text{and}\quad
\abs{\lambda_2}^d\cdot g_-\circ\phi_2 = f_-.
\end{equation*}

Notice that
\begin{align*}
\abs{\lambda_1}^d\cdot g_+(\phi_1(t)) = f_+(t) 
&\implies \abs{\lambda_1}^d\cdot G(1,\phi_1(t)) = F(1,t)\\
&\implies \abs{\lambda_1}^d\cdot \widetilde{G}(1,\phi_1(t)) = \widetilde{F}(1,t)\\
&\implies \abs{\lambda_1}^d\cdot \widetilde{G}\left(t^{-\frac{1}{\beta}},\frac{\phi_1(t)}{t}\right) = \widetilde{F}\left(t^{-\frac{1}{\beta}},1\right).
\end{align*}
In the last implication, we used the fact that 
\(\widetilde{F}\) and \(\widetilde{G}\) are \(\beta\)-quasihomogeneous of the same degree (Claim \ref{claim: F tilde and G tilde beta-quasihomogeneous of degree d-e}). Letting \(t\to+\infty\), we obtain:
\begin{equation*}
\abs{\lambda_1}^d\cdot
\widetilde{G}\left(0,\lim_{\abs{t}\to+\infty} \frac{\phi_1(t)}{t}\right) = 
\widetilde{F}(0,1)
\end{equation*}
Thus, by Claim \ref{claim: F tilde and G tilde beta-quasihomogeneous of degree d-e},
\begin{equation}
\label{eq: (lambda_1)^d G tilde (...) = F tilde (0,1)}
\widetilde{G}\left(0, 
   \abs{\lambda_1}^{\frac{d\beta}{d-e}} \cdot 
      \lim_{\abs{t}\to+\infty} \frac{\phi_1(t)}{t}\right) 
= \widetilde{F}(0,1).
\end{equation}

Similarly,
\begin{align*}
\abs{\lambda_2}^d\cdot g_-(\phi_2(t)) = f_-(t) 
&\implies \abs{\lambda_2}^d\cdot G(-1,\phi_2(t)) = F(-1,t)\\
&\implies \abs{\lambda_2}^d\cdot \widetilde{G}(-1,\phi_2(t)) = \widetilde{F}(-1,t)\\
&\implies \abs{\lambda_2}^d\cdot \widetilde{G}\left(-t^{-\frac{1}{\beta}},\frac{\phi_2(t)}{t}\right) = \widetilde{F}\left(-t^{-\frac{1}{\beta}},1\right).
\end{align*}
Letting \(t\to+\infty\), we obtain:
\begin{equation*}
\abs{\lambda_2}^d\cdot
\widetilde{G}\left(0,\lim_{\abs{t}\to+\infty} \frac{\phi_2(t)}{t}\right) = 
\widetilde{F}(0,1)
\end{equation*}
Thus, by Claim \ref{claim: F tilde and G tilde beta-quasihomogeneous of degree d-e},
\begin{equation}
\label{eq: (lambda_2)^d G tilde (...) = F tilde (0,1)}
\widetilde{G}\left(0, 
   \abs{\lambda_2}^{\frac{d\beta}{d-e}} \cdot 
      \lim_{\abs{t}\to+\infty} \frac{\phi_2(t)}{t}\right) 
= \widetilde{F}(0,1).
\end{equation}

From (\ref{eq: (lambda_1)^d G tilde (...) = F tilde (0,1)}) and 
(\ref{eq: (lambda_2)^d G tilde (...) = F tilde (0,1)}), we obtain:
\begin{equation*}
\widetilde{G}\left(0, 
   \abs{\lambda_1}^{\frac{d\beta}{d-e}} \cdot 
      \lim_{\abs{t}\to+\infty} \frac{\phi_1(t)}{t}\right) 
=
\widetilde{G}\left(0, 
   \abs{\lambda_2}^{\frac{d\beta}{d-e}} \cdot 
      \lim_{\abs{t}\to+\infty} \frac{\phi_2(t)}{t}\right). 
\end{equation*}
Since \(\widetilde{G}(X,Y) = \sum_{k = 0}^nb_kX^{r(n-k)}Y^{sk}\), it follows that
\begin{equation*}
b_n\cdot\left(
   \abs{\lambda_1}^{\frac{d\beta}{d - e}}\cdot
      \lim_{\abs{t}\to+\infty} \frac{\phi_1(t)}{t}
\right)^{sn}
=
b_n\cdot\left(
   \abs{\lambda_2}^{\frac{d\beta}{d - e}}\cdot
      \lim_{\abs{t}\to+\infty} \frac{\phi_2(t)}{t}
\right)^{sn}.
\end{equation*}
Hence,
\begin{equation*}
\abs{\lambda_1}^{\frac{d\beta}{d - e}}\cdot
      \lim_{\abs{t}\to+\infty} \abs{\frac{\phi_1(t)}{t}}
=
\abs{\lambda_2}^{\frac{d\beta}{d - e}}\cdot
      \lim_{\abs{t}\to+\infty} \abs{\frac{\phi_2(t)}{t}}.
\end{equation*}

Now, we consider the case where \(\lambda < 0\). In this case, we have
\begin{equation*}
\abs{\lambda_1}^d\cdot g_-\circ\phi_1 = f_+
\quad\text{and}\quad
\abs{\lambda_2}^d\cdot g_+\circ\phi_2 = f_-.
\end{equation*}

Notice that
\begin{align*}
\abs{\lambda_1}^d\cdot g_-(\phi_1(t)) = f_+(t) 
&\implies \abs{\lambda_1}^d\cdot G(-1,\phi_1(t)) = F(1,t)\\
&\implies \abs{\lambda_1}^d\cdot (-1)^e\cdot \widetilde{G}(-1,\phi_1(t)) = \widetilde{F}(1,t)\\
&\implies \abs{\lambda_+}^d\cdot (-1)^e\cdot \widetilde{G}\left(-t^{-\frac{1}{\beta}},\frac{\phi_1(t)}{t}\right) = \widetilde{F}\left(t^{-\frac{1}{\beta}},1\right).
\end{align*}
Letting \(t\to+\infty\), we obtain:
\begin{equation*}
\abs{\lambda_1}^d\cdot(-1)^e\cdot
\widetilde{G}\left(0,\lim_{\abs{t}\to+\infty} \frac{\phi_1(t)}{t}\right) = 
\widetilde{F}(0,1)
\end{equation*}
Thus, by Claim \ref{claim: F tilde and G tilde beta-quasihomogeneous of degree d-e},
\begin{equation}
\label{eq: (lambda_1)^d G tilde (...) = (-1)^e F tilde (0,1)}
\widetilde{G}\left(0, 
   \abs{\lambda_1}^{\frac{d\beta}{d-e}} \cdot 
      \lim_{\abs{t}\to+\infty} \frac{\phi_1(t)}{t}\right) 
= (-1)^e\cdot\widetilde{F}(0,1).
\end{equation}

Similarly,
\begin{align*}
\abs{\lambda_2}^d\cdot g_+(\phi_2(t)) = f_-(t) 
&\implies \abs{\lambda_2}^d\cdot G(1,\phi_2(t)) = F(-1,t)\\
&\implies \abs{\lambda_2}^d\cdot \widetilde{G}(1,\phi_2(t)) = (-1)^e\cdot\widetilde{F}(-1,t)\\
&\implies \abs{\lambda_2}^d\cdot \widetilde{G}\left(-t^{-\frac{1}{\beta}},\frac{\phi_2(t)}{t}\right) = (-1)^e\cdot\widetilde{F}\left(-t^{-\frac{1}{\beta}},1\right).
\end{align*}
Letting \(t\to+\infty\), we obtain:
\begin{equation*}
\abs{\lambda_2}^d\cdot
\widetilde{G}\left(0,\lim_{\abs{t}\to+\infty} \frac{\phi_2(t)}{t}\right) = 
(-1)^e\cdot\widetilde{F}(0,1)
\end{equation*}
Thus, by Claim \ref{claim: F tilde and G tilde beta-quasihomogeneous of degree d-e},
\begin{equation}
\label{eq: (lambda_2)^d G tilde (...) = (-1)^e F tilde (0,1)}
\widetilde{G}\left(0, 
   \abs{\lambda_2}^{\frac{d\beta}{d-e}} \cdot 
      \lim_{\abs{t}\to+\infty} \frac{\phi_2(t)}{t}\right) 
= (-1)^e\cdot\widetilde{F}(0,1).
\end{equation}

From (\ref{eq: (lambda_1)^d G tilde (...) = (-1)^e F tilde (0,1)}) and 
(\ref{eq: (lambda_2)^d G tilde (...) = (-1)^e F tilde (0,1)}), we obtain:
\begin{equation*}
\widetilde{G}\left(0, 
   \abs{\lambda_1}^{\frac{d\beta}{d-e}} \cdot 
      \lim_{\abs{t}\to+\infty} \frac{\phi_1(t)}{t}\right) 
=
\widetilde{G}\left(0, 
   \abs{\lambda_2}^{\frac{d\beta}{d-e}} \cdot 
      \lim_{\abs{t}\to+\infty} \frac{\phi_2(t)}{t}\right). 
\end{equation*}
Then, it follows that
\begin{equation*}
\abs{\lambda_1}^{\frac{d\beta}{d - e}}\cdot
      \lim_{\abs{t}\to+\infty} \abs{\frac{\phi_1(t)}{t}}
=
\abs{\lambda_2}^{\frac{d\beta}{d - e}}\cdot
      \lim_{\abs{t}\to+\infty} \abs{\frac{\phi_2(t)}{t}}.
\end{equation*}
\end{proof}

\begin{proposition}
\label{prop: characterization of beta-transitions}
Let \(F,G\in\R[X,Y]\) be \(\beta\)-quasihomogeneous polynomials of degree \(d\geq 1\), none of which being of the form \(c X^d\), and let \(f_+,f_-\) be the height functions of \(F\) and \(g_+,g_-\) the height functions of \(G\). Suppose that \((g_+,g_-)\circ(\lambda,\phi) = (f_+,f_-)\) for some proto-transition \((\lambda,\phi)\). Then \((\lambda,\phi)\) is a \(\beta\)-transition if and only if the following conditions hold:
\begin{enumerate}[label=\roman*.]
\item \(\phi_1\) and \(\phi_2\) are {\it coherent} (i.e. they are either both increasing or both decreasing)
\item None of the polynomials \(F,G\) has \(X\) as a factor, or \(\lambda_1 = \lambda_2\).
\end{enumerate}
\end{proposition}

\begin{proof}
First, suppose that \((\lambda,\phi)\) is a \(\beta\)-transition. Then, we have \begin{equation}
\label{eq: limit condition for beta-transitions}
\abs{\lambda_1}^\beta\cdot\lim_{\abs{t}\to+\infty}\frac{\phi_1(t)}{t}
=
\abs{\lambda_2}^\beta\cdot\lim_{\abs{t}\to+\infty}\frac{\phi_2(t)}{t}.
\end{equation}
Since \(\abs{\lambda_1} > 0\) and \(\abs{\lambda_2} > 0\), it follows that 
\(\lim_{\abs{t}\to+\infty} \phi_1(t)/t\) and \(\lim_{\abs{t}\to+\infty} \phi_2(t)/t\) have the same sign. And since \(\phi_1\) and \(\phi_2\) are monotonic, this implies that they are coherent. Hence, condition (i) is satisfied.

Since \((g_+,g_-)\circ(\lambda,\phi) = (f_+,f_-)\), where \((\lambda,\phi)\) is a proto-transition, the multiplicity of \(X\) as a factor of \(F\) is equal to the multiplicity of \(X\) as a factor of \(G\) (Claim \ref{claim: e_F = e_G}). Let us denote by \(e\) the multiplicity of \(X\) both as a factor of \(F\) and as a factor of \(G\). Since none of the polynomials \(F,G\) is of the form \(cX^d\), we have \(e < d\). So, by Claim \ref{claim: limit condition, for e < d},
\begin{equation}
\label{eq: limit condition, for e < d}
\abs{\lambda_1}^{\frac{d\beta}{d - e}}
   \lim_{\abs{t}\to+\infty} \abs{\frac{\phi_1(t)}{t}} = 
\abs{\lambda_2}^{\frac{d\beta}{d - e}}
   \lim_{\abs{t}\to+\infty} \abs{\frac{\phi_2(t)}{t}}.
\end{equation}
Since the limits \(\lim_{\abs{t}\to+\infty} \phi_1(t)/t\) and 
\(\lim_{\abs{t}\to+\infty} \phi_2(t)/t\) have the same sign, it follows that
\begin{equation}
\label{eq: limit condition, for e < d - without absolute value}
\abs{\lambda_1}^{\frac{d\beta}{d - e}}
   \lim_{\abs{t}\to+\infty} \frac{\phi_1(t)}{t} = 
\abs{\lambda_2}^{\frac{d\beta}{d - e}}
   \lim_{\abs{t}\to+\infty} \frac{\phi_2(t)}{t}.
\end{equation}

Since the limits \(\lim_{\abs{t}\to+\infty} \phi_1(t)/t\) and 
\(\lim_{\abs{t}\to+\infty} \phi_2(t)/t\) are both nonzero (because \(\phi_1\) and \(\phi_2\) are bi-Lipschitz), it follows from equations 
(\ref{eq: limit condition for beta-transitions}) and (\ref{eq: limit condition, for e < d - without absolute value}) that
\begin{equation}
\label{eq: ratio of powers of abs(lambda)}
\frac{\abs{\lambda_1}^{\frac{d\beta}{d-e}}}{\abs{\lambda_1}^\beta} = 
\frac{\abs{\lambda_2}^{\frac{d\beta}{d-e}}}{\abs{\lambda_2}^\beta}.
\end{equation}
Equivalently,
\begin{equation}
\label{eq: simplified equation powers of lambda}
\abs{\lambda_1}^\frac{e\beta}{d-e} = \abs{\lambda_2}^\frac{e\beta}{d-e}.
\end{equation}
Furthermore, this equality holds if and only if \(e = 0\) or 
\(\abs{\lambda_1} = \abs{\lambda_2}\). And since \(\lambda_1\) and 
\(\lambda_2\) have the same sign, this is equivalent to condition (ii). Since (\ref{eq: simplified equation powers of lambda}) actually holds, condition (ii) is satisfied.

Now, in order to prove the converse, suppose that conditions (i) and (ii) hold. Since \((g_+,g_-)\circ(\lambda,\phi) = (f_+,f_-)\), where \((\lambda,\phi)\) is a proto-transition, (\ref{eq: limit condition, for e < d}) still holds for this part of the argument. Thus, condition (i) implies (\ref{eq: limit condition, for e < d - without absolute value}). On the other hand, as we have just proved, condition (ii) is equivalent to (\ref{eq: ratio of powers of abs(lambda)}). Since we are assuming that condition (ii) is satisfied, (\ref{eq: ratio of powers of abs(lambda)}) holds. From (\ref{eq: limit condition, for e < d - without absolute value}) and (\ref{eq: ratio of powers of abs(lambda)}), we obtain (\ref{eq: limit condition for beta-transitions}). Therefore, \((\lambda,\phi)\) is a \(\beta\)-transition.
\end{proof}

\begin{corollary}
\label{cor: Consequences of the characterization of beta-transitions}
Let \(F,G\in\R[X,Y]\) be \(\beta\)-quasihomogeneous polynomials of degree \(d\geq 1\), none of which being of the form \(c X^d\), and let \(f_+,f_-\) be the height functions of \(F\) and \(g_+,g_-\) the height functions of \(G\). 
Also, let \(\beta = r/s\), where \(r > s > 0\) and \(\gcd(r,s) = 1\).
Suppose that \((g_+,g_-)\circ(\lambda,\phi) = (f_+,f_-)\) for some proto-transition \((\lambda,\phi)\). Then, we have:
\begin{enumerate}[label = (\alph*)]
\item If\/ \(r\) is even or \(s\) is odd, then there exists a \(\beta\)-transition 
\((\tilde{\lambda},\tilde{\phi})\) such that
\((g_+,g_-)\circ(\tilde{\lambda},\tilde{\phi}) = (f_+,f_-)\).
\item If\/ \(s\) is even, then there exists 
\(\tilde{\phi} = (\tilde{\phi}_1,\tilde{\phi}_2)\),
with \(\tilde{\phi}_1\) and \(\tilde{\phi}_2\) coherent, such that
\((g_+,g_-)\circ(\lambda,\tilde{\phi}) = (f_+,f_-)\).
\end{enumerate}
\end{corollary}

\begin{proof}
Let \(F(X,Y) = \sum_{k=0}^n a_kX^{d-rk}Y^{sk}\) and 
\(G(X,Y) = \sum_{k=0}^n b_k X^{d-rk}Y^{sk}\), with \(a_n,b_n\neq 0\), 
\(n\geq 1\). (In the proof of Claim \ref{claim: e_F = e_G} we showed that the upper limit of summation \(n\) is the same for \(F\) and \(G\), provided that 
\((g_+,g_-)\circ(\lambda,\phi) = (f_+,f_-)\) for some proto-transition 
\((\lambda,\phi)\). Also, we have \(n\geq 1\) because none of the polynomials \(F,G\) is of the form \(cX^d\).) Now, we proceed to the proof of items (a) and (b).
\begin{enumerate}[label=(\alph*)]
\item {\bf Case 1.} \(r\) is even

In this case, we have
\(f_-(t) = (-1)^d\cdot f_+(t)\) and \(g_-(t) = (-1)^d\cdot g_+(t)\).
In fact, 
\begin{equation*}
f_-(t) = \sum_{k = 0}^n a_k\cdot(-1)^{d-rk}\cdot t^{sk} 
= (-1)^d\cdot\sum_{k=0}^n a_kt^{sk} = (-1)^d\cdot f_+(t).
\end{equation*}
The same reasoning, with \(f\) replaced by \(g\), gives the other equation.

By hypothesis, there exists a proto-transition \((\lambda,\phi)\) such that 
\((g_+,g_-)\circ(\lambda,\phi) = (f_+,f_-)\). We claim that 
\((g_+,g_-)\circ(\tilde\lambda,\tilde\phi) = (f_+,f_-)\), where 
\(\tilde\lambda = (\lambda_1, \lambda_1)\) and \(\tilde\phi = (\phi_1,\phi_1)\).
In fact, if \(\lambda > 0\) then \(\abs{\lambda_1}^d\cdot g_+\circ\phi_1 = f_+\) and hence 
\begin{equation*}
\abs{\lambda_1}^d\cdot g_-(\phi_1(t))  
= (-1)^d\cdot \abs{\lambda_1}^d\cdot g_+(\phi_1(t))
= (-1)^d\cdot f_+(t)
= f_-(t),
\end{equation*}
so we also have \(\abs{\lambda_1}^d\cdot g_-\circ\phi_1 = f_-\).
If \(\lambda < 0\) then \(\abs{\lambda_1}^d\cdot g_-\circ\phi_1 = f_+\) and hence
\begin{equation*}
\abs{\lambda_1}^d\cdot g_+(\phi_1(t))  
= (-1)^d\cdot \abs{\lambda_1}^d\cdot g_-(\phi_1(t))
= (-1)^d\cdot f_+(t)
= f_-(t),
\end{equation*}
so we also have \(\abs{\lambda_1}^d\cdot g_+\circ\phi_1 = f_-\).
By Proposition \ref{prop: characterization of beta-transitions}, 
\((\tilde\lambda,\tilde\phi)\) is a \(\beta\)-transform.\\

\noindent
{\bf Case 2.} \(r\) and \(s\) are both odd

In this case, we have \(f_-(t) = (-1)^d\cdot f_+(-t)\) and \(g_-(t) = (-1)^d\cdot g_+(-t)\).
In fact, 
\begin{equation*}
f_-(t) = \sum_{k = 0}^n (-1)^{d-rk}\cdot a_k\cdot t^{sk} 
= (-1)^d\cdot\sum_{k=0}^n (-1)^k a_kt^{sk} = (-1)^d\cdot f_+(-t).
\end{equation*}
The same reasoning, with \(f\) replaced by \(g\), gives the other equation.

By hypothesis, there exists a proto-transition \((\lambda,\phi)\) such that 
\((g_+,g_-)\circ(\lambda,\phi) = (f_+,f_-)\). We claim that 
\((g_+,g_-)\circ(\tilde\lambda,\tilde\phi) = (f_+,f_-)\), where 
\(\tilde\lambda = (\lambda_1, \lambda_1)\) and 
\(\tilde\phi(t) = (\phi_1(t),-\phi_1(-t))\).
In fact, if \(\lambda > 0\) then \(\abs{\lambda_1}^d\cdot g_+\circ\phi_1 = f_+\) and hence 
\begin{equation*}
\abs{\lambda_1}^d\cdot g_-(-\phi_1(-t))  
= (-1)^d\cdot \abs{\lambda_1}^d\cdot g_+(\phi_1(-t))
= (-1)^d\cdot f_+(-t)
= f_-(t),
\end{equation*}
so we also have \(\abs{\lambda_1}^d\cdot g_-(-\phi_1(-t)) = f_-(t)\).
If \(\lambda < 0\) then \(\abs{\lambda_1}^d\cdot g_-\circ\phi_1 = f_+\) and hence
\begin{equation*}
\abs{\lambda_1}^d\cdot g_+(-\phi_1(-t))  
= (-1)^d\cdot \abs{\lambda_1}^d\cdot g_-(\phi_1(-t))
= (-1)^d\cdot f_+(-t)
= f_-(t),
\end{equation*}
so we also have \(\abs{\lambda_1}^d\cdot g_+(-\phi_1(-t)) = f_-(t)\).
By Proposition \ref{prop: characterization of beta-transitions}, 
\((\tilde\lambda,\tilde\phi)\) is a \(\beta\)-transform.

\item Suppose that \(s\) is even. Then, we have
\(g_+(-t) = g_+(t)\) and \(g_-(-t) = g_-(t)\).
In fact,
\begin{equation*}
g_+(-t) = \sum_{k=0}^n b_k(-t)^{sk} = \sum_{k=0}^n b_k t^{sk} = g_+(t)
\end{equation*}
and\\
\begin{equation*}
g_-(-t) = \sum_{k=0}^n (-1)^{d-rk}\cdot b_k(-t)^{sk} 
= \sum_{k=0}^n (-1)^{d-rk}\cdot b_k t^{sk}
= g_-(t).
\end{equation*}

By hypothesis, there exist a proto-transition \((\lambda,\phi)\) such that
\((g_+,g_-)\circ (\lambda,\phi) = (f_+,f_-)\). We claim that 
\((g_+,g_-)\circ (\lambda,\overline{\phi}) = (f_+,f_-)\), where
\(\overline{\phi} = (\phi_1,-\phi_2)\). In fact, if \(\lambda > 0\) then 
\(\abs{\lambda_2}^d\cdot g_-\circ \phi_2 = f_-\) and hence 
\(\abs{\lambda_2}^d\cdot g_-\circ (-\phi_2) = f_-\)
(because \(g_-\) is an even function). If \(\lambda < 0\) then 
\(\abs{\lambda_2}^d\cdot g_+\circ \phi_2 = f_-\) and hence 
\(\abs{\lambda_2}^d\cdot g_+\circ (-\phi_2) = f_-\)
(because \(g_+\) is an even function). 

Finally, notice that \(\phi_1\) is coherent with either \(\phi_2\) or \(-\phi_2\). If \(\phi_1\) and \(\phi_2\) are coherent, we take \(\tilde\phi = \phi\). If \(\phi_1\) and \(-\phi_2\) are coherent, we take \(\tilde\phi = \overline{\phi}\).
\end{enumerate}
\end{proof}

\begin{theorem}
\label{thm: Sufficient conditions for R-semialg. Lip. equivalence}
Let \(F,G\) be \(\beta\)-quasihomogeneous polynomials of degree \(d\geq 1\), none of which being of the form \(cX^d\), and let \(f_+,f_-\) be the height functions of \(F\) and \(g_+,g_-\) the height functions of \(G\). Also, let \(\beta = r/s\), where \(r > s > 0\) and 
\(\gcd(r,s) = 1\). Suppose that
\((g_+,g_-)\circ(\lambda,\phi) = (f_+,f_-)\)
for some proto-transition \((\lambda,\phi)\). 
If any of the following conditions is satisfied then \(F\) and \(G\) are
\({\cal R}\)-semialgebraically Lipschitz equivalent:
\begin{enumerate}[label=(\alph*)]
\item \(r\) is even or \(s\) is odd
\item \(\lambda_1 = \lambda_2\)
\item None of the polynomials \(F,G\) has \(X\) as a factor.
\end{enumerate}
\end{theorem}

\begin{proof}
If \(r\) is even or \(s\) is odd then, by Corollary \ref{cor: Consequences of the characterization of beta-transitions}, there exists a \(\beta\)-transition
\((\tilde\lambda,\tilde\phi)\) such that
\((g_+,g_-)\circ(\tilde\lambda,\tilde\phi) = (f_+,f_-)\). 
Then \(G\circ \Phi = F\), where \(\Phi\) is the inverse \(\beta\)-transform of 
\((\tilde\lambda,\tilde\phi)\). 

Now, assume that either (b) or (c) holds. If \(s\) is odd, then condition (a) is satisfied, and therefore \(F\) and \(G\) are \({\cal R}\)-semialgebraically Lipschitz equivalent, as we have just proved. If \(s\) is even then, by Corollary \ref{cor: Consequences of the characterization of beta-transitions}, there exists \(\tilde\phi = (\tilde\phi_1,\tilde\phi_2)\), with \(\tilde\phi_1\) and 
\(\tilde\phi_2\) coherent, such that 
\((g_+,g_-)\circ(\lambda,\tilde\phi) = (f_+,f_-)\). 
Since we are assuming that either (b) or (c) holds, Proposition \ref{prop: characterization of beta-transitions} guarantees that 
\((\lambda,\tilde\phi)\) is a \(\beta\)-transition. Hence, \(G\circ\Phi = F\),
where \(\Phi\) is the inverse \(\beta\)-transform of \((\lambda,\tilde\phi)\).
\end{proof}

Next, we present an example in which condition (b) of Theorem \ref{thm: Sufficient conditions for R-semialg. Lip. equivalence} is satisfied whereas conditions (a) and (c) are not.\\

\begin{example}
Let\/ \(F(X,Y) = XY^4 + 8X^4Y^2 + 16X^7\) and\/ 
\(G(X,Y) = XY^4 + 18X^4Y^2 + 81X^7\).
These polynomials are both \(\beta\)-quasihomogeneous of degree \(7\), with \(\beta = 3/2\). In this example, \(r = 3\) and \(s = 2\), so condition (a) of Theorem \ref{thm: Sufficient conditions for R-semialg. Lip. equivalence} is not satisfied. Also, notice that \(X\) is a factor of both \(F\) and \(G\), so condition (c) is not satisfied either. Let us show that condition (b) is satisfied.

The height functions of \(F\) are \(f_+(t) = t^4 + 8t^2 + 16\) and 
\(f_-(t) = -t^4 + 8t^2 - 16\). \(f_+\) has only one critical point \(t = 0\), which has multiplicity \(2\) and is the point where \(f_+\) attains its minimum value: \(16\). \(f_-\) has three distinct critical points: \(-2,0,2\), and its multiplicity symbol is \(((0,-16,0),(2,2,2))\).

The height functions of \(G\) are \(g_+(t) = t^4 + 18t^2 + 81\) and 
\(g_-(t) = -t^4 + 18t^2 - 81\). \(g_+\) has only one critical point \(t = 0\), which has multiplicity \(2\) and is the point where \(g_+\) attains its minimum value: \(81\). \(g_-\) has three distinct critical points: \(-3,0,3\), and its multiplicity symbol is \(((0,-81,0),(2,2,2))\).

By Propositions \ref{prop: only one critical point} and \ref{prop: at least 2 critical points}, there exist bi-Lipschitz functions \(\phi_1,\phi_2\from\R\to\R\) such that
\begin{equation*}
g_+\circ\phi_1 = c f_+ \quad\text{and}\quad g_-\circ\phi_2 = c f_-,
\end{equation*}
where \(c = 81/16\). Thus, \((g_+,g_-)\circ(\lambda,\phi) = (f_+,f_-)\),
where \(\phi = (\phi_1,\phi_2)\) and \(\lambda = (c^{-1/7},c^{-1/7})\).
So condition (b) of Theorem \ref{thm: Sufficient conditions for R-semialg. Lip. equivalence} is satisfied, and therefore \(F\) and \(G\) are \({\cal R}\)-semialgebraically Lipschitz equivalent.
\end{example}

\section*{Acknowledgements} This preprint is part of my thesis, which is still a work in progress. I would like to thank my advisor, Prof. Alexandre Fernandes, and also Prof. Lev Birbrair, Prof. Vincent Grandjean, and Prof. Maria Michalska for the continuous support of my PhD studies and research. Besides, I would like to thank them all for their insightful comments and suggestions. Finally, I would like to gratefully acknowledge that throughout my PhD studies I have been supported by a Study Leave from Universidade Estadual de Santa Cruz and I would also like to acknowledge the financial support provided by Funcap (Fundação Cearense de Amparo ao Desenvolvimento Científico e Tecnológico).

\vspace{15pt}
\noindent
S. Alvarez\\
Departamento de Ciências Exatas e Tecnológicas\\
Universidade Estadual de Santa Cruz\\
Campus Soane Nazaré de Andrade, Rodovia Jorge Amado, km 16, Bairro Salobrinho\\
CEP 45662-900. Ilhéus-Bahia\\ 
Brasil\\[5pt]
sergio.qed@gmail.com

\end{document}